\documentclass{article}
\usepackage{amsmath}
\usepackage{latexsym}
\usepackage{amssymb}
\newtheorem{thm}{Theorem}[section]
\newtheorem{la}[thm]{Lemma}
\newtheorem{Defn}[thm]{Definition}
\newtheorem{Remark}[thm]{Remark}

\newtheorem{Example}[thm]{Example}
\newtheorem{Number}[thm]{\!\!}

\newenvironment{rem}{\begin{Remark}\rm}{\end{Remark}}
\newenvironment{numba}{\begin{Number}\rm}{\end{Number}}
\newenvironment{proof}{{\noindent\bf Proof.}}%
                  {\nopagebreak\hspace*{\fill}$\Box$\medskip\par}

\newcommand{\cU}{{\mathcal U}}

\newcommand{\cE}{{\mathcal E}}
\newcommand{\cO}{{\mathcal O}}

\newcommand{\cL}{{\mathcal L}}
\newcommand{\cB}{{\mathcal B}}
\newcommand{\ve}{\varepsilon}
\newcommand{\R}{{\mathbb R}}
\newcommand{\N}{{\mathbb N}}

\newcommand{\C}{{\mathbb C}}
\newcommand{\K}{{\mathbb K}}
\newcommand{\bI}{{\mathbb I}}
\newcommand{\Sph}{{\mathbb S}}
\newcommand{\mto}{\mapsto}
\newcommand{\sub}{\subseteq}
\newcommand{\wt}{\widetilde}

\newcommand{\wb}{\overline}
\newcommand{\cg}{{\mathfrak g}}
\DeclareMathOperator{\id}{id}

\DeclareMathOperator{\graph}{graph}

\DeclareMathOperator{\pr}{pr}
\DeclareMathOperator{\Diff}{Diff}

\DeclareMathOperator{\Evol}{Evol}
\DeclareMathOperator{\Germ}{Germ}
\DeclareMathOperator{\Hol}{Hol}
\DeclareMathOperator{\Fl}{Fl}
\DeclareMathOperator{\ev}{ev}
\DeclareMathOperator{\Repart}{Re}
\DeclareMathOperator{\Impart}{Im}
\DeclareMathOperator{\op}{op}
\DeclareMathOperator{\comp}{comp}
\newcommand{\dl}{{\displaystyle \lim_{\rightarrow}}}
\newcommand{\impl}{\Rightarrow}
\DeclareMathOperator{\Ad}{Ad}

\DeclareMathOperator{\AC}{AC}
\begin{document}
\begin{center}
{\bf\Large Lie groups of real analytic diffeomorphisms\\[1.5mm]
are {\boldmath$L^1$}-regular}\\[4mm]
{\bf Helge Gl\"{o}ckner}\vspace{.5mm}
\end{center}
\begin{abstract}
\hspace*{-4.7mm}Let $M$ be a compact, real analytic manifold
and $G:=\Diff^\omega(M)$
be the Lie group of all real-analytic diffeomorphisms
$\gamma\colon M\to M$,
which is modelled on the locally convex
space $\cg:=\Gamma^\omega(TM)$
of real-analytic vector fields
on~$M$.
Let $\AC([0,1],G)$ be the Lie group
of all absolutely continuous functions
$\eta\colon [0,1]\to G$.
We study flows of time-dependent real-analytic vector fields
on~$M$ which are $\cL^1$ in time,
and their dependence on the time-dependent
vector field.
Notably, we show that the Lie group
$\Diff^\omega(M)$ is $L^1$-regular
in the sense that each
$[\gamma]\in L^1([0,1],\cg)$
has an evolution
$\Evol([\gamma])\in \AC([0,1],G)$
which depends smoothly on~$[\gamma]$.
As tools for the proof,
we develop new results
concerning $L^1$-regularity of
infinite-dimensional Lie groups,
and new results
concerning the continuity and complex analyticity
of non-linear mappings on locally
convex direct limits.\vspace{2mm}
\end{abstract}
{\bf Classification:}
22E65 (primary);
34A12,
%Initial value problems, existence, uniqueness, continuous dependence
%
28B05,
%Vector-valued set functions, measures and integrals
%
34H05,
%Control problems involving ordinary differential equations
%
46E30,
%Spaces of measurable functions (Lp
%-spaces, Orlicz spaces, Köthe function spaces, Lorentz spaces,...
%
46E40.\\[2.3mm]
%Spaces of vector- and operator-valued functions
%
%
%
\section*{Introduction and statement of main results}
For a compact, real analytic manifold~$M$,
we study the parameter-dependence
of flows associated with time-dependent
real-analytic vector fields $\gamma$ on~$M$
which are $\cL^1$ in time,
in a global formulation. The equivalence class
$[\gamma]$ of $\gamma$ in $L^1$ serves as the parameter.\\[2.1mm]
We are working in the setting of infinite-dimensional
calculus known as Keller's $C^k_c$-theory.
Essential concepts and facts concerning this approach
and the corresponding manifolds and Lie groups
are compiled in Appendix~\ref{appA}.\\[3mm]
{\bf The classical case of smooth time dependence.}
To motivate our approach,
let us start with a compact smooth manifold~$M$.
Let $\Gamma(TM)$ be the space of smooth vector fields
on~$M$
and $C^\infty(\bI,\Gamma(TM))$
be the space of time-dependent smooth vector
fields $\gamma\colon \bI\to\Gamma(TM)$
on~$M$, with smooth dependence of $\gamma_t:=\gamma(t)\in \Gamma(TM)$
on $t\in \bI:=[0,1]$.
As usual, we give spaces of smooth functions
and spaces of smooth vector fields
the compact-open $C^\infty$-topology
(as recalled in \ref{vftopo}).
For the differential equation
\[
\dot{y}(t)=\gamma_t(y(t)),
\]
the flow with parameter $\gamma\in C^\infty(\bI,\Gamma(TM))$,
\begin{equation}\label{thfl}
\bI\times \bI\times M\times C^\infty(\bI,\Gamma(TM))\to M,\;\;
(t,t_0,y_0,\gamma)\mto\Fl^\gamma_{t,t_0}(y_0),
\end{equation}
is globally defined, as compactness of~$M$ implies completeness
for each vector field~$\gamma$.
Moreover, the flow~(\ref{thfl})
is smooth (e.g., as a consequence of \cite[Corollary 2.5.20]{GaN}
and \cite[Lemma~1.19\,(a)]{AGS});
actually, we shall only need smoothness for fixed $t_0$,
as established in \cite[Proposition 5.13]{AaS}.
For all $t,t_0\in \bI$, the map
$\Fl^\gamma_{t,t_0}(\cdot)\colon M\to M$
is an element of the group $\Diff(M)$
of $C^\infty$-diffeomorphisms $\psi\colon M\to M$,
which we endow with its natural Fr\'{e}chet--Lie group structure modelled
on $\Gamma(TM)$, as in \cite{Mic,Ham,Mil}.
Fixing the initial time $t_0\in [0,1]$,
one can apply exponential laws twice to deduce that the associated map
\begin{equation}\label{paradigm}
C^\infty(\bI,\Gamma(TM))\to C^\infty(\bI,\Diff(M)),\quad
\gamma\mto (t\mto \Fl^\gamma_{t,t_0}(\cdot))
\end{equation}
is smooth.\footnote{The smooth manifold
structure on $C^\infty(M,M)$
is canonical in the sense of
\cite[Definition~1.17]{AGS} by \cite[Proposition~1.23]{AGS},
so that the open subset
%submanifold
$\Diff(M)\sub C^\infty(M,M)$
inherits an analogous exponential law.
Likewise,
the manifold structure on $C^\infty(\bI,\Diff(M))$
is canonical.}
In this article,
we shall consider differential equations with parameters in a locally convex space~$P$
for which smoothness
of the flow
\[
\bI\times\bI \times M\times P\to M
\]
does not make sense, but
a smoothness property
%analogous to
like~(\ref{paradigm})
is still possible.\\[2.7mm]
{\bf Absolutely continuous functions and
Carath\'{e}odory solutions.}
Our results will involve absolutely continuous functions
with values in locally convex spaces,
and absolutely continuous functions with values in smooth manifolds
which may be infinite-dimensional. The concepts are as follows
(for more details, see Appendix~\ref{secB}).\\[2.3mm]
Recall that a locally convex space~$E$ is called
sequentially complete if each Cauchy sequence
in~$E$ is convergent (see, e.g., \cite{Mil}).
If $E$ is a sequentially complete locally convex space
and $p\in [1,\infty]$,
then
a function
$\eta\colon [a,b]\to E$
is called an \emph{$\AC_{L^p}$-function}
if it is a primitive
of an $\cL^p$-function $\gamma\colon [a,b]\to E$,
\[
\eta(t)=\eta(a)+\int_a^t\gamma(s)\,ds\quad\mbox{for all $\,t\in[a,b]$,}
\]
the integral being understood as a weak $E$-valued
integral with respect to Lebesgue measure
(see \cite[Definition 4.2.8]{Nik}).
The vector-valued $\cL^p$-functions $\gamma\colon \bI\to E$
considered here are Lusin-measurable functions
(as in \cite{FMP} or \cite{Tho}) such that
$\|q\circ \gamma\|_{L^p}<\infty$ for each continuous
seminorm~$q$ on~$E$ (see \cite[Definition 4.1.10]{Nik}).
The concept of Lusin measurability is recalled
in \ref{close-to-borel}.
Cf.~\cite{Lew} for a different, but
overlapping setting
based on a notion of integrability in seminorm.\\[2.3mm]
If~$M$ is a $C^1$-manifold modelled on~$E$
(see, e.g., \cite{Res,GaN,Nee,Sme}), then a
function $\eta\colon [a,b]\to M$
is called $\AC_{L^p}$
if there is a subdivision $a=t_0<\cdots< t_m=b$
such that $\eta([t_{j-1},t_j])$
is contained in the domain $U_\phi$
of a chart $\phi\colon U_\phi\to V_\phi\sub E$
of~$M$ and $\phi\circ \eta|_{[t_{j-1},t_j]}$
is an $E$-valued $\AC_{L^p}$-function
for all $j\in\{1,\ldots,m\}$ (see \cite[Definition 4.2.20]{Nik}).
The $\AC_{L^1}$-functions are called \emph{absolutely continuous}
and we also write $\AC$ in place of $\AC_{L^1}$;
each $\AC_{L^p}$-function is absolutely continuous.\\[2.3mm]
Absolutely continuous
solutions to differential equations
are called \emph{Carath\'{e}odory
solutions}; we recall this concept in~\ref{numba-ode}
and \ref{def-ode-mfd-2}
in the appendix. See~\cite{GaH}
for further information, notably for fundamentals like
local existence of solutions, local uniqueness,
maximal solutions, and flows;
also, see~\cite{MeR} (with an emphasis on differential
equations in Fr\'{e}chet spaces)
and~\cite{Nik}. For the classical
case of differential equations
in Banach spaces and finite-dimensional spaces,
an exposition of Carath\'{e}odory solutions
can be found in the book~\cite{Sch}.\\[2.7mm]
{\bf Main results concerning parameter-dependence
of flows.}
If~$G$ is a Lie group modelled on~$E$
(as in~\cite{Mil}), then the set $\AC_{L^p}(\bI,G)$
of all $G$-valued $AC_{L^p}$-functions on $\bI$
is a Lie group (see~\cite[Proposition 4.2.27]{Nik}).
The Lie group structure can be obtained
as in the classical case of the Lie group $C(\bI,G)$
of continuous $G$-valued functions on~$\bI$
and the Lie groups $C^k(\bI,G)$ of $G$-valued
$C^k$-functions for $k\in \N_0\cup\{\infty\}$
(cf.\ \cite{Mil,PaS,GCX,Nee,GaN}).\\[2.3mm]
We are interested in the Lie group $\Diff^\omega(M)$
of real-analytic diffeomorphisms of a compact,
real analytic manifold~$M$, as in
\cite{KaM}, \cite{DaS}
and previous work of Leslie~\cite{Les}.
It is modelled on the locally convex space $\Gamma^\omega(TM)$
of real-analytic vector fields
on~$M$,
which is a so-called Silva space (a concept recalled in~\ref{muchi}).\\[4mm]
{\bf Theorem~A.}
\emph{Let $M$ be a compact, real-analytic manifold,
and $p\in [1,\infty]$.
Then the following holds.}
\begin{itemize}
\item[\rm(a)]
\emph{For each $\cL^p$-map $\gamma\colon \bI\to\Gamma^\omega(TM)$,
the differential equation}
\[
\dot{y}(t)=\gamma(t)(y(t))
\]
\emph{satisfies local existence and local uniqueness
of Carath\'{e}odory solutions,
and the maximal Carath\'{e}odory solution $\eta_{t_0,y_0}$
to the initial
value problem $\dot{y}(t)=\gamma(t)(y(t))$, $y(t_0)=y_0$
is defined on all of $\bI$, for all $(t_0,y_0)\in \bI\times M$.
Moreover, $\eta_{t_0,y_0} \!\in\!  \AC_{L^p}(\bI,M)$.
Write $\Fl_{t,t_0}^\gamma(y_0)\!:=\!\eta_{t_0,y_0}(t)$.}
\item[\rm(b)]
\emph{For each $\gamma$ as in {\rm(a)} and all $t,t_0\in \bI$, the map
$\Fl_{t,t_0}^\gamma\colon M \to M$
is a real-analytic diffeomorphism of~$M$.}
\item[\rm(c)]
\emph{For all $\gamma$ as in {\rm(a)} and $t_0\in \bI$,
the map $\bI\to\Diff^\omega(M)$, $t\mto \Fl_{t,t_0}^\gamma$
is absolutely continuous $($and actually an $\AC_{L^p}$-map$)$.}
\item[\rm(d)]
\emph{For each $t_0\in \bI$, the mapping}
\begin{equation}\label{the-essence}
L^p(\bI,\Gamma^\omega(TM))\to \AC_{L^p}(\bI,\Diff^\omega(M)),
\;\; [\gamma]\mto(t\mto \Fl^\gamma_{t,t_0})
\end{equation}
\emph{is well defined and $C^\infty$.}\vspace{2mm}
\end{itemize}
To establish Theorem~A,
we shall prove that the Lie group $\Diff^\omega(M)$
is \emph{$L^1$-regular}
in a sense we shall presently recall;
we then show that the theorem
follows from the $L^1$-regularity.
Actually, we can use the same argument
to deduce an analogue of Theorem~A
for another class of diffeomorphism
groups whose $L^1$-regularity is
already known:\\[2.3mm]
If $M$ is a paracompact, finite-dimensional
smooth manifold, let $\Diff_c(M)$
be the Lie group of all smooth
diffeomorphisms $\phi\colon M\to M$
which are compactly supported in the sense
that $\{x\in M\colon \phi(x)\not=x\}$
has compact closure in~$M$ (cf.\ \cite{Mic, Sm0}).
It is modelled on the space $\Gamma_c(TM)$
of compactly supported smooth vector fields,
endowed with the usual locally convex vector topology
making it the locally convex direct limit
of the spaces $\Gamma_K(TM)$ of smooth vector
fields $X\colon M\to TM$
supported in compact set $K\sub M$,
endowed with the compact-open $C^\infty$-topology.
As shown in \cite[Theorem B]{MeR},
the Lie group $\Diff_c(M)$ is $L^1$-regular.
Re-using the proof of Theorem~A,
we shall observe:\\[4mm]
{\bf Theorem B.}
\emph{Let $M$ be a paracompact, finite-dimensional smooth
manifold and $p\in [1,\infty]$.
Then the following holds.}
\begin{itemize}
\item[\rm(a)]
\emph{For each $\cL^p$-map $\gamma\colon \bI\to\Gamma_c(TM)$,
the differential equation}
\[
\dot{y}(t)=\gamma(t)(y(t))
\]
\emph{satisfies local existence and local uniqueness
of Carath\'{e}odory solutions,
and the maximal Carath\'{e}odory solution $\eta_{t_0,y_0}$
to the initial
value problem $\dot{y}(t)=\gamma(t)(y(t))$, $y(t_0)=y_0$
is defined on all of $\bI$, for all $(t_0,y_0)\in \bI\times M$.
Moreover, $\eta_{t_0,y_0}\!\in\! \AC_{L^p}(\bI,M)$.
Write $\Fl_{t,t_0}^\gamma(y_0)\!:=\!\eta_{t_0,y_0}(t)$.}
\item[\rm(b)]
\emph{For each $\gamma$ as in {\rm(a)} and all $t,t_0 \in \bI$, we
have $\Fl_{t,t_0}^\gamma\in\Diff_c(M)$.}
\item[\rm(c)]
\emph{For each $\gamma$ as in {\rm(a)} and each $t_0\in\bI$, the map
$\bI\to\Diff_c(M)$, $t\mto \Fl_{t,t_0}^\gamma$
is~$AC_{L^p}$.}
\item[\rm(d)]
\emph{For each $t_0\in\bI$, the mapping
$L^p(\bI,\Gamma_c(TM))\!\to \! \AC_{L^p}(\bI,\Diff_c(M))$,\linebreak
$[\gamma]\!\mto \! (t\mto \Fl^\gamma_{t,t_0})$ is well defined
and smooth.}\vspace{2mm}
\end{itemize}
Theorem~B may be of largest interest in the case of a compact smooth manifold~$M$.
Then $\Gamma_c(TM)=\Gamma(TM)$
is the space of all smooth vector fields and $\Diff_c(M)=\Diff(M)$
the Lie group of all $C^\infty$-diffeomorphisms of~$M$,
so that Theorem~B\,(d) furnishes a direct analogue of~(\ref{paradigm}).\\[3mm]
{\bf Regularity properties of infinite-dimensional Lie groups.}
If $G$ is a Lie group
modelled on a locally convex space,
then $G$ gives rise
to a smooth left action $G\times TG\to TG$,
$(g,v)\mto g.v$ on its tangent bundle
$TG$ via left translation,
$g.v:=T\lambda_g(v)$ with $\lambda_g\colon G\to G$,
$h\mto gh$.
Let $e\in G$ be the neutral element
and assume that the
modelling space~$E$ of~$G$ is sequentially
complete.
If $\gamma\colon \bI\to \cg$
is an $\cL^p$-map to the Lie algebra
$\cg:=T_eG\cong E$ of~$G$, then the initial value problem
\[
\dot{y}(t)=y(t).\gamma(t),\quad y(0)=e
\]
has at most one Carath\'{e}odory solution
$\eta\colon \bI\to G$
(see \cite[Lemma~4.3.4\,(iii)]{Nik});
the map $\Evol([\gamma]):=\eta\in AC_{L^p}(\bI,G)$
is called the \emph{evolution} of $\gamma$.
If each $\gamma\in\cL^p(\bI,\cg)$ has an evolution,
then $G$ is called \emph{$L^p$-semiregular}.
If $G$ is $L^p$-semiregular and
\[
\Evol\colon L^p(\bI,\cg)\to C(\bI,G)
\]
is smooth as a map to the Lie group
$C(\bI,G)$ of continuous $G$-valued maps
on~$\bI$ (or, equivalently, to the Lie group
$\AC_{L^p}(\bI,G)$ of $G$-valued $\AC_{L^p}$-maps),
then $G$ is called \emph{$L^p$-regular} (cf.\ \cite[Definition 4.3.7]{Nik};
the equivalence is shown in \cite[Theorem 4.3.9]{Nik}).\\[2.3mm]
Let $k\in\N_0\cup\{\infty\}$.
If $\Evol(\gamma)$
exists for each $\gamma\in C^k([0,1],G)$
and the time-$1$-map
$C^k(\bI,\cg)\to G$, $\gamma\mto \Evol(\gamma)(1)$
is smooth, then~$G$ is called \emph{$C^k$-regular}.\footnote{The latter holds if
and only if $\Evol\colon C^k(\bI,\cg)\to
C^{k+1}(\bI,G)$ is smooth (cf.\ \cite[Theorem~A]{SEM})
and hence if and only if
$\Evol\colon C^k(\bI,\cg)\to
C(\bI,G)$ is smooth
(because the inclusion mapping $C^{k+1}(\bI,G)\to C(\bI,G)$
and the point evaluation $C(\bI,G)\to G$, $\eta\mto\eta(1)$
are smooth).}
See~\cite{SEM} for these concepts and \cite{Han},
where an in-depth study of $C^k$-regularity is provided.
The $C^\infty$-regular Lie groups
are simply called \emph{regular}.
For Lie groups
with sequentially complete modelling
spaces, the concept of regularity was introduced by John Milnor~\cite{Mil}.
It remains an open problem whether every such Lie group
is regular.\\[2.3mm]
It is clear from the definitions that the
following implications hold for a Lie group $G$ modelled
on a sequentially complete locally convex space:\\[2.3mm]
$L^1$-regular $\impl$ $L^p$-regular $\impl$
$L^\infty$-regular $\impl$ $C^0$-regular $\impl$ $C^k$-regular $\impl$
regular\\[2.3mm]
(cf.\ also \cite[Theorem~A]{MeR}).
Regularity is a central concept in infinite-dimensional
Lie
theory; see \cite{Mil}, \cite{KaM},
\cite{SEM}, \cite{GaN}, \cite{Han}, and \cite{Nee}
for further information.\\[2.3mm]
As shown by Hanusch~\cite{Han},
every $C^0$-regular Lie group~$G$ is \emph{locally $\mu$-convex}
in the sense of~\cite{SEM},
i.e., for each chart $\phi\colon U_\phi\to V_\phi$ around~$e$ with $\phi(e)=0$
and continuous seminorm $p$ on the modelling space~$E$
of~$G$, there exists a continuous
seminorm~$q$ on~$E$ such that,
for each $n\in\N$, the
product $g_1\cdots g_n$ is in $U_\phi$ and
\[
p(\phi(g_1\cdots g_n))\leq\sum_{j=1}^nq(\phi(g_j)),
\]
for all $g_1,\ldots,g_n\in U_\phi$ with $\sum_{j=1}^nq(\phi(g_j))<1$.
The Trotter Product Formula,
\[
\lim_{n\to\infty}(\exp_G(v/n)\exp_G(w/n))^n=\exp_G(v+w)\quad
\mbox{for all $v,w\in\cg$,}
\]
and the Commutator Formula
hold in each $L^\infty$-regular Lie group~$G$
(cf.\ \cite[Theorem~I]{MeR}), which can be useful, e.g.,
in representation theory (see \cite{NaS}).
Subsequent work showed that $C^0$-regularity
suffices for the conclusion~\cite{Tro}.\\[3mm]
{\bf The main result: {\boldmath$L^1$}-regularity of
{\boldmath$\Diff^\omega(M)$}.}
By the preceding, it is rewarding to establish as
strong regularity properties as possible
for a given Lie group.
So far, $C^1$-regularity of $\Diff^\omega(M)$ was known~\cite{DaS}.
The next theorem implies Theorem~A.\\[2.7mm]
{\bf Theorem C.}
\emph{For every compact, real-analytic manifold~$M$,
the Lie group $\Diff^\omega(M)$ is $L^1$-regular.
In particular, $\Diff^\omega(M)$ is $C^0$-regular.}\\[4mm]
{\bf Methods.}
To prove Theorem~C,
we first construct $\Evol$ on a $0$-neighbourhood
in
$L^1(\bI,\Gamma^\omega(TM))$
and prove its continuity.
To perform the construction,
we use tools for the proof of $L^p$-semiregularity
provided in Section~\ref{sec-semi}.
To establish continuity of~$\Evol$, we use
the following theorem
concerning non-linear functions
on subsets of locally convex direct limits.
It involves global Lipschitz continuity
in a strong sense:\\[4mm]
{\bf Definition.}
Let $E$ and $F$ be locally convex spaces
and $U\sub E$ be a subset.
We say that a function $f\colon U\to F$
is \emph{Lipschitz continuous}
if, for each continuous seminorm~$p$ on~$F$,
there exists a continuous seminorm~$q$ on~$E$
such that
\[
p(f(x)-f(y))\leq q(x-y)\quad\mbox{for all $\,x,y\in U$.}\vspace{2mm}
\]
{\bf Theorem D.}
\emph{Let $E_1\sub E_2\sub\cdots$ be locally convex spaces such that
each inclusion map $E_n\to E_{n+1}$
is continuous and linear.
Give $E:=\bigcup_{n\in\N}E_n$
the vector space structure making each~$E_n$
a vector subspace, and endow~$E$
with the locally convex direct limit topology.
Let $U_n\sub E_n$ be an open convex subset
for each $n\in\N$ such that $U_1\sub U_2\sub\cdots$;
then $U:=\bigcup_{n\in\N}U_n$ is open in~$E$.
If $F$ is a locally convex space and $f\colon U\to F$
a function such that
$f_n:=f|_{U_n}$ is Lipschitz continuous on $U_n\sub E_n$
for each $n\in\N$,
then~$f$ is continuous on $U\sub E$.}\\[4mm]
To complete the proof of Theorem~C,
it suffices to show that continuity of
$\Evol$ implies its smoothness
for an $L^1$-semiregular Lie group modelled
on a sequentially complete locally convex space.
An analogous result for $C^0$-semiregular Lie groups
(requiring integral completeness)
and related situations was obtained by Hanusch.\footnote{Compare 1)
in Theorem~4 and 1) in Corollary~9
in~\cite{Han}.}
We show:\\[3mm]
{\bf Theorem E.}
\emph{Let $G$ be a Lie group modelled
on a sequentially complete locally convex space,
with Lie algebra~$\cg=T_eG$.
Let $p\in [1,\infty]$. If $G$
is $L^p$-semiregular and $\Evol\colon L^p([0,1],\cg)\to C([0,1],G)$
is continuous at~$0$, then $G$ is $L^p$-regular.}\\[4mm]
{\bf General context.}
Differential equations with measurable
right-hand sides are frequently used in control theory
(see, e.g., \cite{Son}).
Likewise, time-dependent left-invariant vector
fields on \emph{finite-dimensional}
Lie groups are common in geometric control theory.\\[2.3mm]
For older results concerning
continuous and differentiable
parameter dependence of flows
for differential equations
with measurable right-hand sides, see \cite{SvM,KaS} and the references therein.
They do not require the right-hand side
to be smooth or analytic in the space variable,
and the conclusions are limited to continuity of the flow
(for fixed $t_0$),
or joint continuity of its differentials in space variables
and parameters.
Flows for time-dependent
analytic vector fields on finite-dimensional
real-analytic manifolds were
studied systematically in works by Jafarpour and Lewis
(like \cite{JL1}),
including the case of non-compact
manifolds for which vector fields
need not be complete and the maximal flows need
not be globally defined.
Germs around a point play a major role in their approach.
The global perspective provided by Theorem~A
is complementary to the findings in
these works,
as well as the smoothness property
in part~(d) of the theorem.\\[2.3mm]
Our studies are also complementary
to the approach of chronological calculus,
where flows are studied via their action
on spaces of smooth functions (cf.\ \cite{CRO}).
Results like Theorem~A are not available there.\\[3mm]
{\bf Another application of Theorem~D.}
We mention that Theorem~D also has consequences
for infinite-dimensional holomorphy,
more specifically
the theory of analytic functions on
locally convex direct limits.
Notably, Dahmen's Theorem
(see \cite[Theorem~A]{Dah}) can be generalized by means
of Theorem~D, as follows:\\[3mm]
{\bf Theorem~F.}
\emph{Let $E_1\sub E_2\sub\cdots$ be complex locally convex spaces
such that each inclusion map $E_n\to E_{n+1}$
is continuous and linear.
Give $E:=\bigcup_{n\in\N}E_n$
the complex vector space structure making each~$E_n$
a vector subspace, and endow~$E$
with the locally convex direct limit topology,
which we assume Hausdorff.
Let $U_n\sub E_n$ be an open convex subset
for each $n\in\N$ such that $U_1\sub U_2\sub\cdots$;
then $U:=\bigcup_{n\in\N}U_n$ is open in~$E$.
If $F$ is a complex locally convex space and $f\colon U\to F$
a function such that
$f_n:=f|_{U_n}$ is complex analytic on $U_n\sub E_n$
for each $n\in\N$, with bounded image $f_n(U_n)\sub F$,
then~$f$ is complex analytic.}\\[4mm]
Dahmen~\cite{Dah} had to assume that each~$E_n$
is a normed space, that the norms are chosen such that
each inclusion map $E_n\to E_{n+1}$
has operator norm~$\leq 1$,
and that each~$U_n$ is the open unit ball in~$E_n$.\\[2.1mm]
Theorem~F actually implies a \emph{characterization} of complex analyticity
for mappings on open subsets of locally convex direct limits
(see Theorem~\ref{characteri});
the author is grateful to Rafael Dahmen (Karlsruhe Institute
of Technology) for discussions
which led to this application.\\[2.3mm]
\noindent{\bf Acknowledgement.}
Supported by Deutsche Forschungsgemeinschaft (DFG),
project GL 357/9-1.
\section{Preliminaries and notation}
We write $\N:=\{1,2,\ldots\}$
and $\N_0:=\N\cup\{0\}$.
All topological vector spaces
are\linebreak
assumed Hausdorff, unless the contrary is stated.
The word
``locally convex topological vector space''
will be abbreviated as ``locally convex space.''
If $q\colon E\to[0,\infty[$ is a seminorm
on a real or complex vector space~$E$,
let $B^q_\ve(x):=\{y\in E\colon q(y-x)<\ve\}$
be the open ball of radius $\ve>0$
around $x\in E$.
If $(E,\|\cdot\|)$
is a normed space and the norm is understood, we also write
$B^E_\ve(x)$ in place of $B^{\|\cdot\|}_\ve(x)$.
A locally convex space~$E$ is called \emph{sequentially complete}
if every Cauchy sequence in~$E$ converges in~$E$.
For basic concepts and notation concerning calculus and analytic
functions in locally convex spaces, see Appendix~\ref{appA}.
For vector-valued $\cL^p$-functions,
vector-valued absolutely
continuous functions and Carath\'{e}odory
solutions to differential equations in
locally convex spaces, see Appendix~\ref{secB}.
All manifolds or Lie groups are modelled
on locally convex spaces which may be infinite-dimensional.
If $G$ is a Lie group with Lie algebra
$\cg:=L(G):=T_eG$, then the map
$I_g\colon G\to G$, $h\mto ghg^{-1}$ is smooth for each
$g\in G$
and induces a Lie algebra automorphism
(and homeomorphism)
$\Ad_g:=T_eI_g$ of~$\cg$.
If $f\colon G\to H$ is a smooth group
homomorphism between Lie groups,
we let $L(f):=T_ef\colon T_eG\to T_eH$
be the corresponding continuous Lie algebra
homomorphism from $L(G)$ to $L(H)$.
Thus $\Ad_g=L(I_g)$.
If~$X$ is a topological
space and $K\sub X$ a subset, then a subset
$U\sub X$ is called a \emph{neighbourhood of~$K$ in~$X$}
if $K\sub U^0$ holds for the interior~$U^0$
of~$U$ in~$X$.\\[2.3mm]
The concepts of $L^p$-regularity and $L^p$-semiregularity
for a Lie group~$G$ were already explained in the introduction,
and need not be repeated.
A \emph{right evolution} to $[\gamma]\in L^1([0,1],\cg)$
is an absolutely continuous function $\eta\colon [0,1]\to G$
such that $\dot{\eta}=[t\mto \gamma(t).\eta(t)]$
and $\eta(0)=e$, using the right action
\[
TG\times G\to TG,\quad (v,g)\mto v.g:=T\rho_g(v)
\]
with $\rho_g\colon G\to G$, $h\mto hg$.
Right evolutions are unique if they exist
and we write $\Evol^r([\gamma]):=\eta$ (cf.\ (33) in \cite[Lemma~5.12]{MeR}
and \cite[Lemma~4.3.4]{Nik}).
If $[\gamma]\in L^p([0,1],\cg)$,
then $\Evol^r([\gamma])\in \AC_{L^p}([0,1],G)$
(cf.\ \cite[Lemma~4.1.23]{Nik}).
We shall use the following fact.
\begin{la}\label{la-left-right}
Let $G$ be a Lie group modelled on a sequentially
complete locally convex space, with Lie algebra
$\cg:=T_eG$. Let $p\in [1,\infty]$.
Then we have:
\begin{itemize}
\item[\rm(a)]
$[\gamma]\in L^p([0,1],\cg)$ has a right evolution
if and only if $-[\gamma]$ has a left\linebreak
evolution.
In this case,
\[
\Evol^r([\gamma])=\Evol(-[\gamma])^{-1},
\]
the inverse in the mapping group $C([0,1],G)$.
\item[\rm(b)]
$G$ is $L^p$-semiregular
if and only if each $[\gamma]\in L^p([0,1],\cg)$
has a right\linebreak
evolution.
\item[\rm(c)]
$G$ is $L^p$-regular if and only if
each $[\gamma]\in L^p([0,1],\cg)$
has a right evolution and $\Evol^r\colon
L^p([0,1],\cg)\to C([0,1],G)$ is smooth.
\end{itemize}
\end{la}
\begin{proof}
(a) can be proved like~(35)
in \cite[Lemma~5.12]{MeR}
(cf.\ (4.15) in \cite[Lemma~4.3.3]{Nik}).
Assertions (b) and~(c) follow immediately from~(a).
\end{proof}
\begin{numba}\label{muchi}
An ascending sequence
$E_1\sub E_2\sub\cdots$
of locally convex spaces $E_n$ is called a \emph{direct sequence
of locally convex spaces} if $E_n$ is a vector subspace of $E_{n+1}$
for each $n\in\N$ and the inclusion map $E_n\to E_{n+1}$
is continuous linear. We shall always endow
\[
E:=\bigcup_{n\in\N}E_n
\]
with the unique vector space structure making each $E_n$
a vector subspace.
The \emph{locally convex direct limit
topology} on $E$ is the finest (not necessarily Hausdorff)
locally convex vector space topology on~$E$
which makes each inclusion map $E_n\to E$ continuous;
we then write $E=\dl\,E_n$.\vspace{-1.1mm}
If each bounded subset of~$E$ is contained in some~$E_n$
and bounded in~$E_n$, then
$E=\bigcup_{n\in\N}E_n$ is called a \emph{regular}
direct limit.
Recall that a \emph{Silva space} (or (DFS)-space)
is a locally convex space~$E$ which can be written as $E=\dl\,E_n$\vspace{-.9mm}
for a direct sequence of Banach spaces~$E_n$, such that all
bonding maps $E_n\to E_{n+1}$ are compact operators.
If~$E$ is Silva, then $E=\bigcup_{n\in\N}E_n$ is
Hausdorff, complete, and regular
(see \cite{Flo}, \cite{FaW}).
Let~$F_n$ be the closure of $E_n$ in $E_{n+1}$.
We endow~$F_n$ with the norm induced by the norm of~$E_{n+1}$.
Let $B^{E_n}_1(0)$ be the open unit ball in~$E_n$
and $K_n:=\wb{B^{E_n}_1(0)}$
be its closure in~$E_{n+1}$.
Then~$K_n$ is compact and metrizable and hence
separable. The countable union
\[
\bigcup_{m\in\N} m\,K_n
\]
is dense in $F_n$ (as it contains~$E_n$)
and is separable, whence also~$F_n$
is separable. We now have a direct sequence
\[
E_1\sub F_1\sub E_2\sub F_2\sub\cdots
\]
of locally convex spaces. Thus $E=\dl\, F_n$\vspace{-.5mm}
and also this locally convex direct limit is regular.
By \cite[Corollary~3.11]{FMP},
we get that
\[
L^1(I,E)=\dl \, L^1(I,F_n)\vspace{-1.1mm}
\]
as a locally convex space.
As a consequence,
$L^1(I,E)$ also is the locally convex direct limit of the direct sequence
\[
L^1(I,E_1)\sub L^1(I,F_1)\sub L^1(I,E_2)\sub\cdots
\]
and also the locally convex direct limit
of any subsequence thereof. Thus
\begin{equation}\label{L1-is-DL}
L^1(I,E)=\dl\, L^1(I,E_n).\vspace{-1.1mm}
\end{equation}
\end{numba}
We recall the theorem on Lipschitz-continuous
dependence of fixed points on parameters
(see, e.g., \cite[Proposition 2.3.26\,(b)]{GaN}).
This well-known fact will allow Theorem~D
to be applied in the proof of Theorem~C.
\begin{la}\label{par-lip}
Let $(P,d_P)$ be a metric space, $(X,d)$
be a complete metric space with $X\not=\emptyset$
and $f\colon P\times X\to X$ be a function.
Given $p\in P$, write $f_p:=f(p,\cdot)\colon X\to X$.
Assume there exist $\theta\in [0,1[$ and $L\in[0,\infty[$
such that
\begin{eqnarray*}
(\forall p\in P)\,(\forall x,y\in X) & & \;\;\;\;\hspace*{.41mm}
d(f_p(x),f_p(y))\leq\theta \,d(x,y)
\;\;\mbox{and}\\
(\forall p,q\in P)\,(\forall x\in X) & & d(f(p,x),f(q,x))\leq L\, d_P(p,q).
\end{eqnarray*}
Then $f_p$ has a unique fixed point $x_p\in X$ for
each $p\in P$, and
\[
d(x_p,x_q)\;\leq\;\frac{L}{1-\theta}\, d_P(p,q)
\;\;\;\mbox{for all $\, p,q\in P$.}
\]
\end{la}
\begin{proof}
Banach's Fixed Point Theorem (the Contraction Mapping
Principle) yields existence and uniqueness of
the fixed point~$x_p$ of $f_p$. For $p,q\in P$,
the \emph{a priori} estimate, applied to the contraction~$f_p$,
gives $d(x_p,x_q)\leq\frac{1}{1-\theta}\,d(f_p(x_q),x_q)$.
Now $d(f_p(x_q),x_q)=d(f_p(x_q),f_q(x_q))=
d(f(p,x_q),f(q,x_q))\leq L\,d_P(p,q)$.
\end{proof} 
\section{Proof of Theorem~D}
Given $x\in U$, let us show that $f$ is continuous at~$x$.
We have $x\in U_m$ for some $m\in\N$; after passing to the subsequence
$E_m\sub E_{m+1}\sub\cdots$, we may assume that $x\in U_1$.
After replacing $U_n$ with $U_n-x$ for each $n\in \N$ and $f$ with
$g\colon U-x\to F$, $g(y):=f(x+y)$, we may assume that $x=0$.
Let~$p$ be a continuous seminorm on~$F$ and $\ve>0$.
For each $n\in\N$,
there exists a continuous seminorm~$q_n$ on~$E_n$ such that
\[
p(f_n(z)-f_n(y))\leq q_n(z-y)\quad\mbox{for all $\,y,z\in U_n$.}
\]
The ascending union
\[
V:=\bigcup_{n\in\N}\sum_{k=1}^n \big(B^{q_k}_{\ve 2^{-k}}(0)\cap 2^{-k}U_k\big)
\]
is an open $0$-neighbourhood in~$E$,
as~$V$ is convex, $0\in V$ and $V\cap E_m$
is open in~$E_m$ for each $m\in \N$.
Moreover,
\[
\sum_{k=1}^n 2^{-k}U_k\sub
\sum_{k=1}^n 2^{-k}U_n
\sub U_n\quad\mbox{for each $n\in\N$,}
\]
as $U_n$ is convex and $0\in U_n$.
If $y\in V$, then $y=\sum_{k=1}^n z_k$
for some $n\in\N$ and elements $z_k\in B^{q_k}_{\ve 2^{-k}}(0)\cap 2^{-k}U_k$
for $k\in\{1,\ldots, n\}$. Abbreviate
\[
y_k:=\sum_{j=1}^k z_j\;\,\mbox{for $k\in\{0,1,\ldots, n\}$;}
\]
thus $y_0=0$ and $y_n=y$. Moreover, $\{y_{k-1},y_k\}\sub U_k$ for all
$k\in\{1,\ldots, n\}$ and
$q_k(y_k-y_{k-1})=q_k(z_k)<\ve \, 2^{-k}$.
Hence
\begin{eqnarray*}
p(f(y)-f(0)) &\leq& \sum_{k=1}^n p(f(y_k)-f(y_{k-1}))
=\sum_{k=1}^n p(f_k(y_k)-f_k(y_{k-1}))\\
&\leq& \sum_{k=1}^n q_k(y_k-y_{k-1})
\leq \sum_{k=1}^n \ve \, 2^{-k} <\ve.
\end{eqnarray*}
Thus $p(f(y)-f(0))<\ve$ for all $y\in V$, and so
$f$ is continuous at~$0$. $\,\square$
\section{Proof of Theorem~E}
We shall use the following fact
(see Lemma~4.1.23 in \cite{Nik}
and its proof):
\begin{numba}\label{varnik}
Let $X$ be a topological
space, $E$ and $F$
be locally convex spaces
and\linebreak
$f \colon X\times E\to F$
be a continuous map such that $f(x,\cdot)\colon E\to F$
is linear for each $x\in X$.
If $a<b$ are real numbers, $p\in[1,\infty]$
and $\eta\colon [a,b]\to X$
is a continuous map, then $f\circ(\eta,\gamma)\in \cL^p([a,b],F)$
for each $\gamma\in\cL^p([a,b],E)$ and~the~map
\[
L^p([a,b],E)\to L^p([a,b],F),\quad [\gamma]\mto [f\circ(\eta,\gamma)]
\]
is continuous and linear.
\end{numba}
Another simple lemma is helpful, which uses
notation as in~\ref{at-point}.
\begin{la}\label{Mvalueddd}
Let $W$ be an open subset of a locally convex space~$E$
and $f\colon W\to M$ be a continuous map
to a $C^1$-manifold~$M$ modelled on a locally
convex space~$F$.
For $(x,y)\in TW=W\times E$,
let $I_{x,y}\sub\R$
be an open $0$-neighbourhood such that $x+ty\in W$
for all $t\in I_{x,y}$.
If
\[
\kappa_{x,y}\colon I_{x,y}\to M,\quad t\mto f(x+ty)
\]
is differentiable at~$0$ for all $(x,y)\in W\times E$
and the map
\[
W\times E\to TM,\quad (x,y)\mto\dot{\kappa}_{x,y}(0)
\]
is continuous, then $f$ is $C^1$ and $Tf(x,y)=\dot{\kappa}_{x,y}(0)$
for all $(x,y)\in W\times E$.
\end{la}
\begin{proof}
Since $W$ can be covered with open subsets
whose images under $f$ are contained in a chart domain,
we may assume that $f(W)\sub U$
for a chart $\phi\colon U\to V\sub F$ of~$M$.
By~\ref{chain-pw}, $\phi\circ f$ has all directional derivatives
and these are
\[
d(\phi\circ f)(x,y)=\frac{d}{dt}\Big|_{t=0}
\phi(\kappa_{x,y}(t))=d\phi\big(\dot{\kappa}_{x,y}(0)\big)
\]
and hence continuous in $(x,y)\in W\times E$.
Thus $\phi\circ f$ is $C^1$ and hence so is~$f$, by the Chain Rule
for $C^1$-functions, with
$Tf(x,y)=T\phi^{-1}\big(T(\phi\, \circ f)(x,y)\big)
=T\phi^{-1}(T\phi (\dot{\kappa}_{x,y}(0)))=\dot{\kappa}_{x,y}(0)$.
\end{proof}
{\bf Proof of Theorem~E.}
Since $G$ is $L^p$-semiregular,
$L^p([0,1],\cg)$ can be made a group with neutral element~$[0]$
and group multiplication given by
\[
[\gamma]\odot[\eta]:=[\Ad(\Evol([\eta]))^{-1}.\gamma]+[\eta]
\]
(in view of \cite[Lemma~4.3.4\,(i)]{Nik},
we can proceed as in \cite[Definition~5.34]{MeR}).
By~\ref{varnik},
the right translation
\[
\rho_{[\eta]}\colon L^p([0,1],\cg)\to L^p([0,1],\cg),\quad
[\gamma]\mto [\gamma]\odot [\eta]
\]
is a continuous affine-linear map
for each $[\eta]\in L^p([0,1],\cg)$
and hence a homeomorphism (and $C^\infty$-diffeomorphism),
as also its inverse is a right translation.
For $\beta:=\Evol([\eta])$, the right translation
\[
\rho_\beta\colon C([0,1],G)\to C([0,1],G),\quad
\zeta\mto \zeta\beta
\]
is continuous. Since $\Evol$ is continuous at~$0$, we deduce from
\[
\Evol=\rho_\beta\circ \Evol\circ \, \rho_{[\eta]^{-1}}
\]
that $\Evol\colon \!\! L^p([0,1],\cg)\!\hspace{-.3mm}\to\hspace{.3mm}\!\! C([0,1],G)$
is continuous at~$\hspace{-.3mm}[\eta]$.
\hspace{-.3mm}Thus~$\Evol$ is~continuous.\\[2.3mm]
Let $\phi\colon U\to V$ be a $C^\infty$-diffeomorphism
from an open $e$-neighbourhood $U\sub G$
onto an
open $0$-neighbourhood $V\sub \cg$ such that $\phi(e)=0$
and $d\phi|_\cg=\id_\cg$.
Let $Y\sub U$ be an open identity neighbourhood
such that $Y=Y^{-1}$ and $YY\sub U$. Then $Z:=\phi(Y)\sub V$
is an open $0$-neighbourhood.
The map
\[
\mu\colon Z\times Z\to V,\quad
(x,y)\mto\phi(\phi^{-1}(x)\phi^{-1}(y))
\]
is smooth.
The set $C([0,1],Z)$ is an open $0$-neighbourhood in
$C([0,1],\cg)$; the set $C([0,1],Y)$
is an open identity neighbourhood in $C([0,1],G)$,
and the map
\[
\Phi\colon C([0,1],Y)\to C([0,1],Z),\quad \eta\mto\phi|_Y\circ \eta
\]
is a $C^\infty$-diffeomorphism (cf.\ \cite{GCX,GaN,PaS,AGS}).
As $\Evol\colon L^p([0,1],\cg)\to C([0,1],G)$
is continuous and takes $[0]$ to the constant function~$e$,
there exists an open $0$-neighbourhood
$\Omega\sub L^p([0,1],\cg)$
with $\Evol(\Omega)\sub C([0,1],Y)$.
For each $[\zeta]\in\Omega$,
the function $\Evol([\zeta])\colon [0,1]\to G$
is absolutely continuous with
$\Evol([\zeta])^{\cdot}=[t\mto\Evol([\zeta]).\zeta(t)]$.
Using~\ref{chain-abs}, we see that $\theta:=\phi\circ \Evol([\zeta])\colon [0,1]\to\cg$
is absolutely continuous with
$(\phi\circ\Evol([\zeta]))'=[t\mto d_2\mu(\theta(t),0,\zeta(t))]$.
Thus
\begin{equation}\label{zetataugamma}
\theta(t)=\int_0^t\theta'(s)\,ds
=\int_0^td_2\mu(\theta(s),0,\zeta(s))\,ds\mbox{ for all $t\in [0,1]$.}
\end{equation}
We now show that
the directional derivative
$d(\Phi\circ \Evol|_\Omega)(0,[\gamma])$ exists in $C([0,1],\cg)$
and is given~by
\begin{equation}\label{theder0}
\frac{d}{d\tau}\Big|_{\tau=0}\Phi(\Evol(\tau[\gamma]))=
\left(t\mto \int_0^t\gamma(s)\,ds\right),
\end{equation}
for each $\gamma\in \cL^p([0,1],\cg)$.
There is $\rho>0$ such that $\tau[\gamma]\in\Omega$
for all $\tau\in\,]{-\rho},\rho[$.
Applying (\ref{zetataugamma})
with $\tau\gamma$ in place of~$\zeta$,
we obtain
\[
\phi(\Evol(\tau[\gamma])(t))=
\tau \int_0^td_2\mu(\theta_\tau(s),0,\gamma(s))\,ds\;\;
\mbox{for all $\, t\in [0,1]$,}
\]
with $\theta_\tau:=\phi\circ\Evol(\tau[\gamma])$.
The function
\[
f\colon Z\times \cg\to\cg,\quad (x,y)\mto d_2\mu(x,0,y)-y
\]
is smooth and $I\colon L^p([0,1],\cg)\to C([0,1],\cg)$, $[\zeta]\mto I([\zeta])$
with
\[
I([\zeta])(t)=\int_0^t\zeta(s)\, ds
\]
is continuous linear.
For $0\not=\tau\in\,]{-\rho},\rho[$, using $\phi\circ \Evol(0)=0$
we see that
\[
\Delta_\tau := \frac{\phi\circ \Evol(\tau [\gamma])-\phi\circ \Evol(0)}{\tau}-I([\gamma])
=\frac{\phi\circ \Evol(\tau\gamma)}{\tau}-I([\gamma])
\]
satisfies
\[
\Delta_\tau(t) \!=\!\!\!
\int_0^t\! d_2\mu(\phi(\Evol(\tau[\gamma])(s)),\gamma(s))-\gamma(s)\,ds
\!= \!\!\! \int_0^t \!f(\phi(\Evol(\tau[\gamma])(s)),\gamma(s))\,ds.
\]
Let $q$ be a continuous seminorm on~$\cg$.
To establish (\ref{theder0}), we show that
\[
q(\Delta_\tau(t))\to 0\quad\mbox{as $\tau\to 0$,}
\]
uniformly in $t\in [0,1]$. The map~$f$
is $C^2$. Since $f(x,\cdot)\colon\cg\to\cg$
is linear for each $x\in Z$, we have $f(x,0)=0$.
Since $\mu(0,y)=y$ for each $y\in Z$,
we have $d_2\mu(0,0,y)=y$ for each $y\in\cg$
and hence $f(0,y)=0$. Now standard arguments
show that there exist
a $0$-neighbourhood $X\sub Z$, a $0$-neighbourhood $O\sub \cg$
and continuous seminorms
$q_1,q_2$ on~$\cg$ such that
\[
q(f(x,y))\leq q_1(x)q_2(y)\quad\mbox{for all $(x,y)\in X\times O$}
\]
see Lemma~\ref{new}.
Due to the linearity in the second argument,
we may assume that $O=\cg$.
Let $\ve>0$.
By continuity of $\Evol$, there exists
$\delta\in \,]0,\rho]$ with
\[
\phi\circ \Evol(\tau[\gamma])
\in C([0,1],X\cap B^{q_1}_\ve(0))
\;\;\mbox{for all $\tau\in \,]{-\delta},\delta[$.}
\]
For all $0\not=z\in \,]{-\delta},\delta[$,
we deduce that
\begin{eqnarray*}
q(\Delta_\tau(t))&\leq &
\int_0^t q(f(\phi(\Evol(\tau[\gamma])(s),\gamma(s)))\,ds\\
&\leq&
\int_0^t\underbrace{q_1(\phi(\Evol(\tau[\gamma])(s))}_{<\ve}
q_2(\gamma(s))\,ds\\
&
\leq & \ve\,\|q_2\circ\gamma\|_{\cL^1}
\leq \ve \, \|q_2\circ\gamma\|_{\cL^p}
=\ve \,\|[\gamma]\|_{L^p,q_2},
\end{eqnarray*}
using H\"{o}lder's inequality for the last estimate.
As the right-hand side is independent of $t\in [0,1]$
and can be made arbitrarily small, the desired
uniform convergence of $q\circ \Delta_\tau$
is established.\\[2.3mm]
By (\ref{theder0}), we have
$\frac{d}{d\tau}\big|_{\tau=0}\Phi(\Evol(\tau[\gamma]))=I([\gamma])$,
whence the curve $\tau\mto\Evol(\tau[\gamma])$
is differentiable at $\tau=0$
in the sense of~\ref{dotty}, with derivative
(as in (\ref{stasta}))~given~by
\[
\frac{d}{d\tau}\Big|_{\tau=0}\Evol(\tau[\gamma])=T\Phi^{-1}
(0, I([\gamma])).
\]
Since $\Evol$ is a homomorphism of groups from
$(L^p([0,1],\cg),\odot)$ to $C([0,1],G)$
with pointwise multiplication, for each $[\eta]\in L^p([0,1],\cg)$
we have
\[
\Evol|_{\Omega[\eta]}=\rho_{\Evol([\eta])}\circ \Evol|_\Omega\circ \rho_{[\eta]^{-1}}|_{\Omega[\eta]}
\]
on the open set $\Omega[\eta]\!:=\!\rho_{[\eta]}(\Omega)\!\sub \! L^p([0,1],\cg)$.
As $\rho_{[\eta]^{-1}}\colon \!L^p([0,1],\cg) \!\to \! L^p([0,1],\cg)$\linebreak
is an affine-linear homeomorphism and given by
\[
[\zeta]\mto [t\mto\Ad(\Evol([\eta])(t))(\zeta(t)-\eta(t))]
\]
(cf.\ \cite[Definition~5.34]{MeR} and \cite[Lemma~4.3.4\,(i)]{Nik}),
we deduce that
\[
\kappa_{[\eta],[\gamma]}(\tau):=\Evol([\eta]+\tau[\gamma])\in C([0,1],G)
\]
for $[\gamma]\!\in\!  L^p([0,1],\cg)$ and real numbers $\tau$ close to~$0$ is differentiable at $\tau=0$,
with
\[
\dot{\kappa}_{[\eta],[\gamma]}(0)
=T\rho_{\Evol([\eta])}T\Phi^{-1}\left(0,I(\Ad^\wedge \circ(\Evol([\eta]),\gamma))\right),
\]
using the smooth map $\Ad^\wedge\colon G\times\cg\to\cg$, $(g,v)\mto\Ad(g)(v)$.
Thus
\begin{equation}\label{enabindu}
\dot{\kappa}_{[\eta],[\gamma]}(0)
=\sigma\left(T\Phi^{-1}\Big(0,
I(\Ad^\wedge \circ(\Evol([\eta]),\gamma))\Big),\Evol([\eta])\right)
\end{equation}
using the smooth right action $T(C([0,1],G))\times C([0,1],G)\to T(C([0,1],G))$,
$(v,\beta)\mto T\rho_\beta(v)$.
Since $\Ad^\wedge$ is linear in the second argument and~$C^\infty$,
the~map
\[
C([0,1],G)\times L^p([0,1],\cg)\to L^p([0,1],\cg),\;\;
(\beta,[\zeta])\mto [\Ad^\wedge\circ (\beta,\zeta)]
\]
is smooth (see \cite[Proposition~4.3.5]{Nik}).
As $\Evol\colon L^p([0,1],\cg)\to C([0,1],G)$
is continuous, (\ref{enabindu}) shows
that $\dot{\kappa}_{[\eta],[\zeta]}\in T(C([0,1],G))$
is continuous in $([\eta],[\zeta])\in L^p([0,1],\cg)\times L^p([0,1],\cg)$.
Thus $\Evol$ is $C^1$, by Lemma~\ref{Mvalueddd},
with $T\Evol([\eta],[\zeta])=\dot{\kappa}_{[\eta],[\zeta]}$
given by~(\ref{enabindu}).
If $\Evol$ is $C^k$ for some $k\in\N$,
then the right-hand side of~(\ref{enabindu}),
and hence $T\Evol([\eta],[\zeta])\in T(C([0,1],G))$,
is a $C^k$-function of $([\eta],[\zeta])$
as before. In view of~\ref{indu-Ck},
this implies that $\Evol$ is $C^{k+1}$.
By the preceding inductive argument,
$\Evol$ is $C^k$ for all $k\in\N$,
and thus $\Evol$ is~$C^\infty$. $\,\square$
\begin{rem}
\hspace*{-.8mm}In Theorem~E,
the evolution map $\Evol\colon \!L^p([0,1],\cg)\!\to \! C([0,1],G)$
is smooth also with respect to the topology~$\cO_{L^1}$
on $L^p([0,1],\cg)$ induced by $L^1([0,1],\cg)$
(which is coarser than the usual $L^p$-topology).
In fact, $\Evol$ is continuous
with respect to $\cO_{L^1}$ by the version of
\cite[Theorem~1.9]{GaH}
described in \cite[Remark~10.21]{GaH}
(as we know from Theorem~E
that~$G$ is $L^p$-regular).
We can now repeat the proof
of Theorem~E as in the case
of an $L^1$-semiregular Lie group,
replacing $L^1([0,1],\cg)$ with
its vector subspace $(L^p([0,1],\cg),\cO_{L^1})$.
\end{rem}
\section{Tools to establish {\boldmath$L^p$-semiregularity}}\label{sec-semi}
We develop tools to allow a given function $\eta\colon [0,1]\to G$
to be recognized as the evolution
of a given $[\gamma]\in L^1([0,1],\cg)$.
The following lemma
will help us to prove Theorem~C.
\begin{la}\label{recog}
Let $G$ be a Lie group
modelled on sequentially complete locally
convex space~$E$, with Lie algebra $\cg:=T_e(G)$.
Let $\gamma\in \cL^1([0,1],\cg)$
and $\eta\colon [0,1]\to G$ be a continuous
function such that $\eta(0)=e$.
We assume that there exists a\linebreak
sequence
$(G_n)_{n\in\N}$ of Lie groups~$G_n$
modelled on sequentially complete locally convex spaces~$E_n$,
with Lie algebras $\cg_n:=T_e(G_n)$,
and a sequence of smooth group homomorphisms
$f_n\colon G\to G_n$ such that
the corresponding Lie algebra homomorphisms
$L(f_n)\!:=\!T_e(f_n)\colon \!\cg\! \to\! \cg_n$
separate points on~$\cg$.
Moreover, assume:
\begin{itemize}
\item[\rm(a)]
$\eta_n:=f_n\circ\eta\in \AC([0,1],G_n)$
holds for each $n\in\N$ and $\eta_n$
is the evolution of
$[L(f_n)\circ\gamma]\in L^1([0,1],\cg_n)$; and
\item[\rm(b)]
There exist a chart $\phi\colon U\to V\sub E$
of~$G$ with $\eta([0,1])\sub U$ and $\phi(e)=0$,
and charts $\phi_n\colon U_n\to V_n\sub E_n$
of~$G_n$ with $\eta_n([0,1])\sub U_n$ and $\phi_n(e)=0$
for $n\in\N$ such that
$f_n(U)\sub U_n$ and there exists a
continuous linear map
$\alpha_n\colon E\to E_n$ such that
$\phi_n\circ f_n|_U=\alpha_n\circ \phi$.
\end{itemize}
Then $\eta\in \AC([0,1],G)$
and $\eta$ is the evolution $\Evol([\gamma])$ of~$[\gamma]$.
\end{la}
\begin{proof}
The sets $W:=(d\phi|_\cg)^{-1}(V)\sub\cg$
and $W_n:=(d\phi_n|_{\cg_n})^{-1}(V_n)\sub\cg_n$
for $n\in\N$ are open $0$-neighbourhoods;
moreover, $\psi:=(d\phi|_\cg)^{-1}\circ \phi\colon U\to W$
and $\psi_n:=(d\phi_n|_{\cg_n})^{-1}\circ\phi_n\colon U_n \to W_n$
are $C^\infty$-diffeomorphisms such that
\begin{equation}\label{better-v}
L(f_n)\circ \psi=\psi_n\circ f_n|_U,
\end{equation}
using that $d\phi_n|_{\cg_n}\circ L(f_n)=\alpha\circ d\phi|_\cg$.
By construction, $d\psi|_\cg=\id_\cg$ and $d\psi_n|_{\cg_n}=\id_{\cg_n}$
for all $n\in\N$.
The sets
$D:=\{(x,y)\in W\times W\colon \psi^{-1}(x)\psi^{-1}(y)\in U\}$
and
\[
D_n:=\{(x,y)\in W_n\times W_n\colon \psi_n^{-1}(x)\psi_n^{-1}(y)\in U_n\}
\]
are open in $W\times W$ and $W_n\times W_n$, respectively.
We consider the smooth functions
\[
\mu\colon D\to W,\quad (x,y)\mto\psi(\psi^{-1}(x)\psi^{-1}(y))
\]
and $\mu_n\colon D_n\to W_n$, $(x,y)\mto\psi_n(\psi_n^{-1}(x)\psi_n^{-1}(y))$.
Let $\lambda_g\colon G\to G$, $h\mto gh$ be left translation
by $g\in G$.
Given $x\in W$, the set $D_x:=\{y\in W\colon (x,y)\in D\}$
is an open $0$-neighbourhood, and we can
consider the $C^\infty$-map
$\lambda_x\colon D_x\to W$,\linebreak
$y\mto\mu(x,y)$.
By construction, $\psi\circ \lambda_g=\lambda_{\psi(g)}\circ \psi$
holds on an open identity neighbourhood in~$G$, entailing that
\begin{equation}\label{rel-tra}
d\psi\circ T\lambda_g|_\cg=
d\lambda_{\psi(g)}\circ T\psi|_\cg=
d\lambda_{\psi(g)}(0,\cdot)\quad\mbox{for all $\,g\in U$.}
\end{equation}
Given $n\in\N$ and $x\in W_n$, we define
$D_x\sub W_n$ and $\lambda_x\colon D_x\to W_n$
analogously, using~$\mu_n$; we obtain
\begin{equation}\label{rel-tra2}
d\psi_n\circ T\lambda_g|_{\cg_n}=
d\lambda_{\psi(g)}(0,\cdot)\quad \mbox{for all $g\, \in U_n$,}
\end{equation}
where $\lambda_g\colon G_n\to G_n$ is left translation by~$g$.\\[2.3mm]
Since $\eta_n=\Evol([L(f_n)\circ\gamma])$,
we have $\dot{\eta}_n=[t\mto T\lambda_{\eta_n(t)}(L(f_n)(\gamma(t))]$.
Using~(\ref{rel-tra2}), we deduce that
\begin{equation}\label{wusch}
\hspace*{-1.4mm}(\psi_n\! \circ \eta_n)'\!\!=\! [t\!\mto\! d\psi_n (T\lambda_{\eta_n(t)}(L(f_n)(\gamma(t)))]
\!=\! [t\!\mto\! d\lambda_{\psi_n(\eta_n(t))}(0,L(f_n)(\gamma(t)))].\!\!
\end{equation}
For $g\in U$, on some open identity neighbourhood of~$G$
we have
\[
L(f_n)\circ \! \lambda_{\psi(g)}\! \circ \psi
=L(f_n)\circ \psi\circ \lambda_g
=\psi_n\circ f_n\! \circ \lambda_g
= \psi_n\circ \lambda_{f_n(g)}\! \circ f_n
=\lambda_{\psi_n(f_n(g))}\circ f_n,
\]
using (\ref{better-v}) for the second equality.
As a consequence, for all $g\in U$,
\begin{equation}\label{all-in-1}
L(f_n)\circ d\lambda_{\psi(g)}(0,\cdot)
=L(f)\circ d\lambda_{\psi(g)}\circ T\psi|_\cg
=d\lambda_{\psi_n(f_n(g)}(0,\cdot)\circ L(f_n).
\end{equation}
Abbreviate
\[
\theta:=\psi\circ \eta\quad\mbox{and}\quad
\theta_n:=\psi_n\circ\eta_n
\]
for $n\in\N$ and define $\Theta\in \AC([0,1],\cg)$ via
\[
\Theta(t):=\int_0^t d_2\mu(\theta(s),0,\gamma(s))\,ds
=\int_0^t d\lambda_{\theta(s)}(0,\gamma(s))\,ds
\quad \mbox{for $\, t\in[0,1]$,}
\]
using partial differentials.
We claim that $\theta=\Theta$.
If this is true, then $\eta=\psi^{-1}\circ\theta$ is absolutely
continuous and
\[
\dot{\eta} = (\psi^{-1}\circ \, \theta)^{\cdot}
\! = [t\mto T\psi^{-1}T\lambda_{\theta(t)}(0,\gamma(t))]
= [t\mto T\lambda_{\eta(t)}(\gamma(t))]
=[t\mto \eta(t).\gamma(t)],
\]
using~(\ref{rel-tra}) for the third equality.
Thus $\eta=\Evol([\gamma])$.\\[2.3mm]
To establish the claim, for $t\in [0,1]$ we observe that
\begin{eqnarray*}
L(f_n)(\Theta(t)) &=&
\int_0^t L(f_n)(d\lambda_{\theta(s)}(0,\gamma(s)))\,ds
=
\int_0^td\lambda_{\theta_n(s)}(0,(L(f_n)\circ\gamma)(s))\,ds\\
&= &\theta_n(t)=\psi_n(f_n(\eta(t)))
=L(f_n)(\psi(\eta(t))
= L(f_n)(\theta(t))
\end{eqnarray*}
for all $n\in\N$,
using (\ref{all-in-1}) for the second equality,
(\ref{wusch}) for the third, and (\ref{better-v})
for the fifth equality.
As the $L(f_n)$ separate points on~$\cg$,
we infer $\Theta(t)=\theta(t)$.
\end{proof}
\begin{rem}\label{single-hom}
We shall apply Lemma~\ref{recog}
in the special case that
$G$ and $H$ are Lie groups
modelled on sequentially complete locally convex spaces and
$f\colon G\to H$ is a smooth
group homomorphism such that $L(f)$ is injective,
using $G_n:=H$ and $f_n:=f$ for all $n\in\N$.
\end{rem}
\begin{rem}\label{recog-left-right}
Of course,
evolutions (i.e., left evolutions)
can be replaced with right evolutions
in Lemma~\ref{recog}.
To see this, we can use
the mechanism described in Lemma~\ref{la-left-right};
or we can simply replace left translations with
right translations in the proof
of Lemma~\ref{recog},
and the natural left action of~$G$ on $TG$ with the natural
right action.
\end{rem}
\section{Local preparations for Theorem~C}\label{sec-loc}
We prove preparatory results
which will be used in the proof of Theorem~C.
Throughout this section, we shall use the following
notation:
\begin{numba}\label{convents}
Let $m\in\N$ and let $\|\cdot\|$
be the maximum norm on~$\C^m$
given by $\|z\|:=\max\{|z_1|,\ldots,|z_m|\}$
for $z=(z_1,\ldots, z_m)\in\C^m$.
If $z_k=x_k+iy_k$ for $k\in\{1,\ldots,m\}$ with $x_k,y_k\in\R$,
write
\[
\Repart(z):=(x_1,\ldots,x_m)\quad\mbox{and}\quad
\Impart(z):=(y_1,\ldots,y_m).
\]
Let $\tau\colon \C^m\to\C^m$, $z\mto\bar{z}:=(\wb{z_1},\ldots,\wb{z_m})$
be complex conjugation.
For $\ve>0$,
consider the open subsets
$W_\ve:=\;]{-2-\ve},2+\ve[^m+\, i\;]{-\ve},\ve[^m$
and
\[
V_\ve\, :=\, \frac{1}{2}W_\ve = \;\Big]{-1-\frac{\ve}{2}},1+\frac{\ve}{2}\Big[^m+\,
i\;\Big]{-\frac{\ve}{2}},\frac{\ve}{2}\Big[^m
\]
of~$\C^m$. Note that for $z=x+iy\in W_\ve$ with $x,y\in\R^m$, we have
\begin{equation}\label{price}
\|z\|\leq\|x\|+\|y\|<2+2\ve.
\end{equation}
Let
$\cL(\C^m)$
be the Banach space space of all $\C$-linear
mappings $\alpha\colon \C^m \to\C^m$,
endowed with the operator norm
$\|\alpha\|_{\op}:=\max_{\|u\|\leq 1}
\|\alpha(u)\|$.
\end{numba}
\begin{numba}
For each open set $U\sub\C^m$ and finite-dimensional complex
Banach space $(E,\|\cdot\|_E)$,
the space $BC(U,E)$ of bounded
continuous functions $\theta\colon U\to E$
is a Banach space with respect to the supremum norm
given by
\[
\|\theta\|_\infty:=\sup\{\|\theta(z)\|_E\colon z\in U\}\;\;
\mbox{for $\,\theta\in BC(U,E)$.}
\]
Choosing $(E,\|\cdot\|_E):=(\C^m,\|\cdot\|)$,
the vector subspace $\Hol_b(U,\C^m)$
of bounded complex analytic functions
is closed in $BC(U,\C^m)$ (as complex analyticity is characterized
by Cauchy's integral formula) and hence a Banach space
with respect to the supremum norm.
Let $BC^1_\C(U,\C^m)\sub \Hol_b(U,\C^m)$
be the vector subspace of all $\theta\in \Hol_b(U,\C^m)$
such that the mapping
$\theta'\colon U\to\cL(\C^m)$
is bounded and thus $\theta'\in BC(U,\cL(\C^m))$.
Then $BC^1_\C(U,\C^m)$ is a Banach space
with respect to the norm given by
\[
\|\theta\|_{BC^1}:=\max\{\|\theta\|_\infty,\|\theta'\|_\infty\}
\]
using $\|\cdot\|$ on~$\C^m$ and $\|\cdot\|_{\op}$ on $\cL(\C^m)$
to calculate the supremum norms.
Completeness follows from the observation that $BC^1_\C(U,\C^m)$ is a closed
topological vector subspace of the corresponding Banach space $BC^1(U,\C^m)$ of real
differentiable maps treated in \cite[Corollary 3.2.12]{Wal}.
Finally, we endow $\Hol(U,\C^m):=C^\infty_\C(U,\C^m)$
with the compact-open
topology \!(which coincides with the
compact-open $C^\infty\!$-topology,\,\cite[Lemma~A.7]{DaS}).
\end{numba}
\begin{numba}
For each $\ve>0$,
let $\Hol_b(V_\ve,\C^m)_\R$
be the set of all $\theta\in\Hol_b(V_\ve,\C^m)$
such that $\theta(V_\ve\cap\R^m)\sub\R^m$
(or, equivalently,
$\theta(\bar{z})=\overline{\theta(z)}$ for all
$z\in V_\ve$).
Let $BC^1_\C(W_\ve,\C^m)_\R$
be the set of all
$\theta\in BC^1_\C(W_\ve,\C^m)$
such that $\theta(W_\ve\cap\R^m)\sub\R^m$.
Then $\Hol_b(V_\ve,\C^m)_\R$ is a closed
real vector subspace of $\Hol_b(V_\ve,\C^m)$
and hence a real Banach space when endowed with the
supremum norm.
Similarly, $BC^1_\C(W_\ve,\C^m)_\R$
is a closed real vector subspace of $BC^1_\C(W_\ve,\C^m)$
and hence a real Banach space
with the induced norm.
Finally, we endow $\Hol(W_\ve,\C^m)_\R:=
\{\theta\in \Hol(W_\ve,\C^m)\colon \theta(W_\ve\cap\R^m)\sub\R^m\}$
with the compact-open topology.
\end{numba}
\begin{numba}\label{hence-lip}
If $\theta\in BC^1_\C(W_\ve,\C^m)_\R$,
then $\theta$ is Lipschitz continuous
with Lipschitz constant $\|\theta'\|_\infty\leq\|\theta\|_{BC^1}$,
since~$W_\ve$ is convex and the map $\theta'$ is bounded
(by a standard application of the Mean Value Theorem;
see, e.g., \cite[Lemma~1.5.3]{GaN}).
\end{numba}
\begin{numba}\label{shortcut}
If $\theta\in BC^1_\C(W_\ve,\C^m)_\R$
and $x+iy\in W_\ve$ with $x,y\in\R^m$,
then
\[
\|\Impart (\theta(x+iy))\|\leq \|\theta\|_{BC^1}\|y\|.
\]
In fact, using that $\Impart(\theta(x))=0$,
the Mean Value Theorem
shows that
\begin{eqnarray*}
\|\Impart(\theta(x+iy))\| &=& \|\Impart(\theta(x+iy)-\theta(x))\|\\
&\leq & \int_0^1 \|\Impart(\theta'(x+ity)(iy))\|\, dt
\leq \|\theta\|_{BC^1}\|y\|
\end{eqnarray*}
as $\|\Impart(\theta'(x+ity)(iy))\|\leq\|\theta'(x+ity)(y)\|
\leq\|\theta'(x+ity)\|_{\op}\|y\|\leq\|\theta\|_{BC^1}\|y\|$.
\end{numba}
\begin{numba}\label{thepq}
If $U\sub\C^m$ is an open set, let
$\Hol_b^\partial(V_\ve,U)_\R$ be the set of all
$\theta\in\Hol_b(V_\ve,\C^m)_\R$
with closure $\overline{\theta(V_\ve)}\sub U$.
For $\ve>0$, we let~$Q_\ve$
be the set of all functions $\gamma\in\cL^1([0,1],BC^1_\C(W_{2\ve},\C^m)_\R)$
such that
\begin{equation}\label{defQeps}
\|\gamma\|_{\cL^1,\|\cdot\|_{BC^1}}<1/2.
\end{equation}
Then $P_\ve:=\{[\gamma]\colon \gamma\in Q_\ve\}$
is an open ball in
$L^1([0,1],BC^1_\C(W_{2\ve},\C^m)_\R)$.
\end{numba}
\begin{numba}
Let $\overline{W_\ve}\sub\C^m$ be the closure
of~$W_\ve$. Then
\[
\Hol_b(V_\ve,\wb{W_\ve})_\R:=\{\theta\in\Hol_b(V_\ve,\C^m)_\R\colon
\theta(V_\ve)\sub\wb{W_\ve}\}
\]
is a closed subset of $\Hol_b(V_\ve,\C^m)_\R$.
As a consequence,
$C([0,1],\Hol_b(V_\ve,\wb{W_\ve})_\R)$
is a closed subset of $C([0,1],\Hol_b(V_\ve,\C^m)_\R)$.
\end{numba}
\begin{la}\label{local-lem}
Let notation be as before, and $\ve>0$.
\begin{itemize}
\item[\rm(a)]
For each $\gamma\in Q_\ve$,
there exists a unique
$\zeta\in C([0,1],\Hol_b(V_\ve,\wb{W_\ve})_\R)$
which solves the integral equation
\begin{equation}
\zeta(t)=\id_{V_\ve}+\int_0^t \gamma(s)\circ \zeta(s)\,ds
\end{equation}
in $\Hol_b(V_\ve,\C^m)_\R$ for all $t\in [0,1]$.
We abbreviate $\Psi^{[\gamma]}_\ve:=\zeta$.
\item[\rm(b)]
The map
\[
P_\ve\to C([0,1],\Hol_b(V_\ve,\C^m)_\R),\quad
[\gamma]\mto \Psi^{[\gamma]}_\ve
\]
is Lipschitz continuous; we have
\[
\|\Psi^{[\gamma_2]}_\ve-\Psi^{[\gamma_1]}_\ve\|_\infty
\, \leq\,  2\,\|[\gamma_2]-[\gamma_1]\|_{L^1,\|\cdot\|_\infty}
\;\;\mbox{for all $\, [\gamma_1],[\gamma_2]\in P_\ve$.}
\]
\item[\rm(c)]
For all $\gamma\in Q_\ve$,
we have $\Psi^{[\gamma]}_\ve\in \AC([0,1],\Hol_b(V_\ve,\C^m)_\R)$.
\item[\rm(d)]
For all $\gamma\in Q_\ve$,
the differential equation $y'(t)=\gamma(t)(y(t))$
on $W_{2\ve}$ satisfies local existence and local uniqueness
of Carath\'{e}ordory solutions.
For all $y_0\in V_\ve$,
the maximal Carath\'{e}odory solution $\eta_{0,y_0}$ to
the initial value problem $y'(t)=\gamma(t)(y(t))$, $y(0)=y_0$
on $W_{2\ve}$ is defined on all of $[0,1]$,
and takes its values in~$W_\ve$.
Moreover, $\eta_{0,y_0}(t)=\Psi^{[\gamma]}_\ve(t)(y_0)$
for all $t\in [0,1]$.
\item[\rm(e)]
For real numbers $\delta \in\; ]0,\ve]$,
let $\rho\colon BC^1_\C(W_{2\ve},\C^m)_\R\to BC^1_\C(W_{2\delta},\C^m)_\R$,
$\theta\mto\theta|_{W_{2\delta}}$ be the restriction map.
Then $\rho\circ\gamma\in Q_\delta$
for each $\gamma\in Q_\ve$,
and $\Psi^{[\rho\circ \gamma]}_\delta(t)=\Psi^{[\gamma]}_\ve(t)|_{V_\delta}$
for each $t\in [0,1]$.
\end{itemize}
\end{la}
\begin{proof}
(a), (b), and (c): The composition map
\[
\comp\colon BC^1_\C(W_{2\ve},\C^m)_\R\times
\Hol_b^\partial(V_\ve,W_{2\ve})_\R\to \Hol_b(V_\ve,\C^m)_\R,
\quad
(\theta,\zeta)\mto \theta\circ\zeta
\]
is continuous, as it is a restriction of the composition map
\[
BC^1(W_{2\ve},\C^m)\times BC^{0,\partial}(V_\ve,W_{2\ve})\to BC^0(V_\ve,\C^m)
\]
in \cite[Lemma~3.3.8]{Wal},
which is continuous.
Moreover, $\comp$ is linear in its first argument.
As a consequence,
we have
\[
\comp\circ\, (\gamma,\eta)\in\cL^1([0,1],\Hol_b(V_\ve,\C^m)_\R)
\]
for all $\gamma\in \cL^1([0,1],BC^1_\C(W_{2\ve},\C^m)_\R)$
and $\eta\in C([0,1],\Hol_b^\partial(V_\ve,W_{2\ve})_\R)$,
by \cite[Lemma~4.1.23]{Nik}.
We therefore obtain a map $\comp_\star\colon$
\[
L^1([0,1],BC^1_\C(W_{2\ve},\C^m)_\R)
\times C([0,1],\Hol_b^\partial(V_\ve,W_{2\ve})_\R)
\!\to\!
L^1([0,1],\Hol_b(V_\ve,\C^m)_\R)
\]
via $([\gamma],\eta)\mto [\comp\circ \,(\gamma,\eta)]$.
The operator
\[
I\colon L^1([0,1],\Hol_b(V_\ve,\C^m)_\R)\to
C([0,1],\Hol_b(V_\ve,\C^m)_\R)
\]
determined by
\[
I([\gamma])(t):=
\int_0^t\gamma(s)\,ds
\]
is continuous and linear, with $\|I\|_{\op}\leq 1$.
Composing, we obtain a map\footnote{We write $\id_{V_\ve}$
also for the constant map $[0,1]\to \Hol_b(V_\ve,\C^m)_\R$,
$t\mto \id_{V_\ve}$.}
\[
g\colon P_\ve\times C([0,1],\Hol_b^\partial(V_\ve,W_{2\ve}))_\R)\to
C([0,1],\Hol_b(V_\ve,\C^m)_\R),
\]
\[
([\gamma],\eta)\mto \id_{V_\ve}+I(\comp_\star([\gamma],\eta)),
\]
and define $f$ as its restriction
to the smaller domain $P_\ve\times C([0,1],\Hol_b(V_\ve,\wb{W_\ve})_\R)$.
By construction,
\[
f([\gamma],\eta)(t)=\id_{V_\ve}+\int_0^t
\gamma(s)\circ\eta(s)\, ds
\]
in $\Hol_b(V_\ve,\C^m)_\R$ for each $t\in [0,1]$.
Then $f$ has image in $C([0,1],\Hol_b(V_\ve,\wb{W_\ve})_\R)$
actually. In fact, given $\gamma\in Q_\ve$ and $\eta\in
C([0,1],\Hol_b(V_\ve,\overline{W_\ve})_\R)$, let $t\in[0,1]$ and
$z=(z_1,\ldots,z_n)\in V_\ve$ be arbitrary.
Then
\[
\Repart(f([\gamma],\eta)(t)(z))=
\Repart(z)+\int_0^t \Repart(\gamma(s)(\eta(s)(z)))\,ds
\]
implies
$\|\Repart(f([\gamma],\eta)(t)(z))\|
\leq\|\Repart(z)\|+\int_0^t\|\gamma(s)\|_\infty\,ds
\leq \|\Repart(z)\|+\|\gamma\|_{\cL^1,\|\cdot\|_{BC^1}}$\linebreak
$<1+\ve/2+1/2<2+\ve$.
Using~\ref{shortcut} and (\ref{price}), we get
\begin{eqnarray*}
\|\Impart(f([\gamma],\eta)(t)(z))\|
&\leq & \|\Impart(z)\|+\int_0^t\|\Impart(\gamma(s)(\eta(s)(z)))\|\,ds\\
&\leq & \|\Impart(z)\|+\int_0^t \|\gamma(s)\|_{BC^1}\|\Impart(\eta(s)(z))\|\,ds\\
&< & \ve/2 + \ve \,\|\gamma\|_{\cL^1,\|\cdot\|_{BC^1}} <\ve.
\end{eqnarray*}
Thus, as required,
\begin{equation}\label{in-ope}
f([\gamma],\eta)(t)(z)\in W_\ve\sub\wb{W_\ve}.
\end{equation}
By the preceding, $f([\gamma],\cdot)$ is a self-map of
the closed subset
\[
C([0,1],\Hol_b(V_\ve,\wb{W_\ve}))_\R)
\]
of the Banach space $C([0,1],\Hol_b(V_\ve,\C^m)_\R)$.
We show that $f([\gamma],\cdot)$
is a contraction with Lipschitz constant $\leq 1/2$.
In fact, for $\eta_1,\eta_2\in C([0,1],\Hol_b(V_\ve,\wb{W_\ve})_\R)$,
we have for all $t\in[0,1]$ and $z\in V_\ve$ that
\begin{eqnarray*}
\|f([\gamma],\eta_2)(t)(z)-f([\gamma],\eta_1)(t)(z)\|
&\leq &\!\!\int_0^t \|\gamma(s)(\eta_2(s)(z))-\gamma(s)(\eta_1(s)(z))\|\,ds\\
&\leq & \!\! \int_0^t\|\gamma(s)\|_{BC^1}\|\eta_2(s)(z)-\eta_1(s)(z)\|\,ds\\
&\leq & \|\gamma\|_{\cL^1,\|\cdot\|_{BC^1}}\|\eta_2-\eta_1\|_\infty\leq
1/2\, \|\eta_2-\eta_1\|_\infty,
\end{eqnarray*}
using the Lipschitz constant from~\ref{hence-lip}
and the estimate $\|\eta_2(s)(z)-\eta_1(s)(z)\|\leq\|\eta_2(s)-\eta_1(s)\|_\infty
\leq\|\eta_2-\eta_1\|_\infty$. Passing to the suprema
in $z$ and in~$t$, we deduce that
$\|f([\gamma],\eta_2)-f([\gamma],\eta_1)\|_\infty\leq 1/2\,
\|\eta_2-\eta_1\|_\infty$, as asserted.\\[2.3mm]
For all $\gamma_1,\gamma_2\in Q_\ve$, note that
\begin{eqnarray*}
\lefteqn{\|f([\gamma_2],\eta)(t)(z)-f([\gamma_1],\eta)(t)(z)\|}\qquad\qquad\\
&\leq & \int_0^t\|\gamma_2(s)(\eta(s)(z))-\gamma_1(s)(\eta(s)(z))\|\,ds\\
&\leq& \int_0^t\|\gamma_2(s)-\gamma_1(s)\|_\infty\,ds
\leq  \|[\gamma_2]-[\gamma_1]\|_{L^1,\|\cdot\|_\infty}
\end{eqnarray*}
for all $\eta\in C([0,1],\Hol_b(V_\ve,\wb{W_\ve})_\R)$,
$t\in [0,1]$ and $z\in V_\ve$.
Passing to the suprema in~$z$ and $t$, we obtain
$\|f([\gamma_2],\eta)-f([\gamma_1],\eta)\|_\infty
\leq \|[\gamma_2]-[\gamma_1]\|_{L^1,\|\cdot\|_\infty}$,
showing that $f(\cdot,\eta)$ is a Lipschitz map from~$P_\ve$ with metric
\[
([\gamma_1],[\gamma_2])\mto \|[\gamma_2]-[\gamma_1]\|_{L^1,\|\cdot\|_\infty}
\]
to the Banach space $C([0,1],\Hol_b(V_\ve,\C^m)_\R)$,
with Lipschitz constant~$1$.
Now Lemma~\ref{par-lip} on Lipschitz continuous parameter dependence
of fixed points
shows that
$f([\gamma],\cdot)$
has a unique fixed point $\Psi^{[\gamma]}_\ve:=\zeta$
in
$C([0,1],\Hol_b(V_\ve,\wb{W_\ve})_\R)$,
which is Lipschitz continuous in $[\gamma]$
with
\[
\|\Psi^{[\gamma_2]}_\ve-\Psi^{[\gamma_1]}_\ve\|_\infty
\leq \frac{1}{1-1/2}\|[\gamma_2]-[\gamma_1]\|_{L^1,\|\cdot\|_\infty}
\]
for all $[\gamma_1],[\gamma_2]\in P_\ve$.
Since $I$ takes values in absolutely continuous functions,
we have $\zeta=f([\gamma],\zeta)\in
\AC([0,1],\Hol_b(V_\ve,\C^m)_\R)$.\\[2.3mm]
(d) By~\ref{hence-lip}, the function $[0,1]\times W_{2\ve}\to\C^m$,
$(t,y)\mto \gamma(t)(y)$ satisfies an $L^1$-Lipschitz condition
in the sense of~\cite[Definition~10.1]{MeR}
(cf.\ also \cite[Definition~22.36]{Sch}).
Hence $y'(t)=\gamma(t)(y(t))$ satisfies
local uniqueness of Carath\'{e}odory solutions (see \cite[Proposition~10.5]{MeR})
and local existence (see \cite[30.9]{Sch}).\\[2.3mm]
Let $y_0\in V_\ve$.
Since $\zeta:=\Psi^{[\gamma]}_\ve\in \AC([0,1],\Hol_b(V_\ve,\C^m)_\R)$
and the evaluation map $\ev_{y_0}\colon \Hol_b(V_\ve,\C^m)_\R\to\C^m$,
$\theta\mto \theta(y_0)$ is continuous and real linear,
we have $h:=\ev_{y_0}\circ \,\Psi^{[\gamma]}_\ve\in \AC([0,1],\C^m)$. Moreover,
applying $\ev_{y_0}$ to the integral equation at time $t\in [0,1]$, we get
\[
h(t)=\ev_{y_0}(\zeta(t))=y_0+\int_0^t\ev_{y_0}(\gamma(s)\circ\zeta(s))\,ds
=y_0+\int_0^t\gamma(s)(h(s))\,ds,
\]
showing that $h$ is a Carath\'{e}odory solution to the initial value problem
$y'(t)=\gamma(t)(y(t))$, $y(0)=y_0$. Thus $\eta_{0,y_0}=h\in \AC([0,1],\C^m)$.
Moreover, $\eta_{0,y_0}(t)=\zeta(t)(y_0)=f([\gamma],\zeta)(t)(y_0)
\in W_\ve$, by~(\ref{in-ope}).\\[2.3mm]
(e) Given $\gamma\in Q_\ve$,
we have $\rho\circ \gamma=\cL^1([0,1],\rho)(\gamma)
\in \cL^1([0,1],BC^1_\C(W_{2\delta},\C^m)_\R)$.
Since $\|(\rho\circ\gamma)(t)\|_{BC^1}=\|\gamma(t)|_{W_{2\delta}}\|_{BC^1}
\leq \|\gamma(t)\|_{BC^1}$ for all $t\in[0,1]$, we have
\[
\|\rho\circ\gamma\|_{\cL^1, \|\cdot\|_{BC^1}}\leq
\|\gamma\|_{\cL^1,\|\cdot\|_{BC^1}}<1/2\vspace{-.9mm}
\]
and thus $\rho\circ \gamma\in Q_\delta$.
Given $y_0\in V_\delta$,
both of the maps $t\mto \Psi^{[\rho\hspace*{.2mm}\circ\gamma]}_\delta(t)(y_0)$
and $t\mto \Psi^{[\gamma]}_\ve(t)(y_0)$
are Carath\'{e}odory solutions to $y'(t)=\gamma(t)(y(t))$, $y(0)=y_0$
in $W_{2\ve}$ and hence coincide (by local uniqueness in~(d)
and \cite[Lemma~3.2]{GaH}).
Thus $\Psi^{[\rho\hspace*{.2mm} \circ\gamma]}_\delta(t)(y_0)=\Psi^{[\gamma]}_\ve(t)(y_0)$
for all $t\in [0,1]$ and $y_0\in V_\delta$.
\end{proof}
Let $B^{\C^m}_\delta(0)$ be the ball around~$0$ in $\C^m$
of radius $\delta>0$, for the maximum norm.
\begin{la}\label{final-loc}
Let $r>0$ and $\alpha\colon \Omega\to W_r$
be a complex analytic function
on an open subset $\Omega\sub TW_r=W_r\times\C^m$
with $W_r\times \{0\}\sub\Omega$ such that
$\alpha(z,0)=z$ for all $z\in W_r$
and
\[
\theta\colon \Omega\to W_r\times W_r,\quad
(z,w)\mto (z,\alpha(z,w))
\]
has open image and is a complex analytic
diffeomorphism onto the image.
Moreover, assume that $\alpha(z,\cdot)'(0)=\id_{\C^m}$
for each $z\in W_r$. Let $\ve\in\,]0,r[$.
Then there exists $\delta_0>0$
such that $W_\ve\times B_{\delta_0}^{\C^m}(0)\sub \Omega$
and, for each open subset $U\sub W_\ve$
and $\delta\in\,]0,\delta_0]$, the open set
$\theta(U\times B_\delta^{\C^m}\!(0))$ contains
the open subset
\[
\bigcup_{z\in U}(\{z\}\times B_{\delta/2}^{\C^m}(z))
\;\;\, \mbox{of $\;W_r\times W_r$.}\vspace{-.9mm}
\]
Thus $\theta^{-1}(z,w)\in
U\times B_\delta^{\C^m}\!(0)$
for all $(z,w)\in U\times\C^m$ such that~$\|w-z\|<\delta/2$.
Moreover, the map
\[
\pr_2\circ \, \theta^{-1}(z,\cdot)|_{B^{\C^m}_{\delta/2}(z)}
\colon B_{\delta/2}^{\C^m}(z)\to\C^m
\]
is Lipschitz continuous with Lipschitz constant~$2$,
for each $z\in U$ $($where we use the projection
$\pr_2\colon\C^m\times\C^m\to\C^m$
onto the second component$)$.
\end{la}
\begin{proof}
The map $h\colon \Omega\to[0,\infty[$,
$(z,w)\mto \|\alpha(z,\cdot)'(w)-\id_{\C^m}\|_{\op}$
is continuous and $h(z,0)=0$
for all $z\in \wb{W_\ve}$.
Thus $\wb{W_\ve}\times\{0\}$ is a compact subset of the open set
$h^{-1}([0,1/2[)$. By the Wallace Theorem~\cite{Key},
there is an open subset $A\sub W_r$ with $\wb{W}_\ve\sub A$
and an open $0$-neighbourhood $B\sub\C^m$
such that $A\times B\sub h^{-1}([0,1/2[)$.
There exists $\delta_0>0$ such that
$B^{\C^m}_{\delta_0}(0)\sub B$.
Then
\[
W_\ve\times B^{\C^m}_{\delta_0}(0)\sub A\times B\sub h^{-1}(([0,1/2[).
\]
Define $\alpha_z\colon B^{\C^m}_{\delta_0}(0)\to\C^m$
for $z\in W_\ve$ via $\alpha_z(w):=\alpha(z,w)$.
Then $\alpha_z$ is complex analytic and $\|\alpha_z'(w)-\id_{\C^m}\|_{\op}\leq 1/2$
for all $w$ in the convex open set $Q:=B^{\C^m}_{\delta_0}(0)$,
whence $\alpha_z-\id_Q$ is Lipschitz continuous with Lipschitz constant~$1/2$.
By the Quantitative Inverse Function Theorem (see \cite[Theorem~2.3.32]{GaN},
also \cite[Lemma 6.1\,(a)]{IMP}), this entails that
\[
\alpha_z(B_\delta^{\C^m}\!(0))\supseteq B^{\C^m}_{\delta/2}(\alpha_z(0))=
B^{\C^m}_{\delta/2}(z)
\quad\mbox{for all $\, \delta\in \,]0,\delta_0]$;}
\]
moreover, $\pr_2(\theta(z,\cdot)^{-1})|_{B^{\C^m}_{\delta/2}(z)}
=\alpha_z^{-1}|_{B^{\C^m}_{\delta/2}(z)}$
is Lipschitz continuous with Lip\-schitz constant~$2$.
If $U\sub W_\ve$ is an open set and $\delta\in\;]0,\delta_0]$,
then $\alpha(U\times B^{\C^m}_\delta\!(0))$ is open in $W_r\times \C^m$
since $\alpha$ is a homeomorphism onto its open image.
Moreover,\linebreak
$\alpha(U\times B^{\C^m}_\delta\!(0))=\bigcup_{z\in U}\alpha_z(B^{\C^m}_\delta\!(0))
\supseteq \bigcup_{z\in U}B^{\C^m}_{\delta/2}(z)$
holds and the union on the right-hand side is open, as a consequence
of the triangle inequality.
\end{proof}
\begin{la}\label{gives-compact}
The restriction mapping
$\rho\colon \Hol_b(V_\ve,\C^m)_\R \to  \Hol_b(V_\delta,\C^m)_\R$,\linebreak
$\theta\mto\theta|_{V_\delta}$
is a compact operator, for all real numbers $0<\delta<\ve$.
\end{la}
\begin{proof}
If $U\sub\C^m$ is open and $W\sub U$ is a relatively compact,
open subset, then the restriction map
\[
\rho_{W,U}\colon \Hol_b(U,\C^m)\to\Hol_b(W,\C^m),\quad
\theta\mto \theta|_W
\]
is a compact operator. To see this,
let $V\sub U$ be a relatively compact, open subset such that $\wb{W}\sub V$.
The restriction map $r\colon \Hol_b(U,\C^m)\to BC^1_\C(V,\C^m)$
is continuous as the restriction
map $C^1(U,\C^m)\to BC^1(V,\C^m)$
is continuous for the compact-open $C^1$-topology.
We now consider the inclusion map $j\colon BC^1_\C(V,\C^m)\to \Hol(V,\C^m)$,
which is continuous linear.
For the unit ball~$B$ in $BC^1_\C(V,\C^m)$,
the image $j(B)$ is equicontinuous
and pointwise bounded, whence $j(B)$
is relatively compact in $C(V,\C^m)$ with the compact-open
topology by\linebreak
Ascoli's Theorem,
and hence also in the closed vector subspace $\Hol(V,\C^m)$.
The restriction map $R\colon \Hol(V,\C^m)\to \Hol_b(W,\C^m)$
is continuous linear. Hence $R(j(B))$ is relatively
compact in $\Hol_b(W,\C^m)$,
whence $R\circ j$ is a compact operator. As a consequence, also
$\rho_{W,U}=R\circ j\circ r$ is a compact operator.\\[2.3mm]
Let $\lambda\colon \Hol_b(V_\ve,\C^m)_\R\to\Hol_b(V_\ve,\C^m)$
be the inclusion map; it is continuous and linear.
By the preceding, $\rho_{V_\delta,V_\ve}\circ\lambda$
is a compact operator. As this map takes its values
in the closed real vector subspace $\Hol_b(V_\delta,\C^m)_\R$
of $\Hol_b(V_\delta,\C^m)$, its corestriction~$\rho$
to a map into this vector subspace is
compact as well.
\end{proof}
If $X$ is a topological space and $(E,\|\cdot\|)$
a normed space, we write $BC(X,E)$ for the space of bounded, continuous
functions $\gamma\colon X\to E$, endowed with the supremum norm
$\|\cdot\|_\infty$. We let $\graph(\gamma)\sub X\times E$
be the graph of~$\gamma$.
\begin{la}\label{fstar-lip}
Let $X$ be a topological space, $(E,\|\cdot\|_E)$
and $(F,\|\cdot\|_F)$ be normed spaces,
$U\sub X\times E$ be a subset and
$f\colon U\to F$ be a bounded, continuous function.
For $x\in X$, write $U_x:=\{y\in E\colon (x,y)\in U\}$;
assume that there exists $L\in [0,\infty[$
such that, for each $x\in X$, the map
\[
f_x:=f(x,\cdot)\colon U_x\to F
\]
is Lipschitz continuous on $U_x\sub E$ with Lipschitz constant~$L$.
Define $D_{f_*}:=\{\gamma\in BC(X,E)\colon\graph(\gamma)\sub U\}$.
For each $\gamma\in D_{f_*}$,
\[
f_*(\gamma)(x):=f(x,\gamma(x))\;\;\mbox{for $\,x\in X$}
\]
then defines a function $f_*(\gamma)\in BC(X,F)$. The map
\[
f_*\colon D_{f_*}\to BC(X,F),\;\;\gamma\mto f_*(\gamma)
\]
is Lipschitz continuous on $D_{f_*}\sub BC(X,E)$,
with Lipschitz constant~$L$.
\end{la}
\begin{proof}
For each $\gamma\in D_{f_*}$,
the map $(\id_X,\gamma)$ is continuous,
whence also its co-restriction $(\id_X,\gamma)|^U$
and $f_*(\gamma)=f\circ (\id_X,\gamma)|^U$ is continuous.
Since $f$ is bounded, also $f_*(\gamma)$ is bounded.
If $\gamma_1,\gamma_2\in D_{f_*}$, we estimate for $x\in X$
\begin{eqnarray*}
\|f_*(\gamma_2)(x)-f_*(\gamma_1)(x)\|_F
&=&\|f(x,\gamma_2(x))-f(x,\gamma_1(x))\|_F\\
&\leq & L\,\|\gamma_2(x)-\gamma_1(x)\|_E\leq L\,\|\gamma_2-\gamma_1\|_\infty.
\end{eqnarray*}
Passing to the supemum in~$x$, we get
$\|f_*(\gamma_2)-f_*(\gamma_1)\|_\infty\leq L\, \|\gamma_2-\gamma_1\|_\infty$.
\end{proof}
\section{Preparations concerning {\boldmath$\Gamma^\omega(TM)$}}\label{sec-gens}
We discuss generalities concerning the vector space
$\Gamma^\omega(TM)$ of real-analytic vector fields
on a compact real-analytic manifold~$M$.
\begin{numba}
Let~$M$ be a compact real-analytic manifold.
Since~$M$ can be embedded in~$\R^N$ for some~$N$
(see \cite{Gra}),
there exists a real-analytic Riemannian metric~$g$
on~$M$. The associated Riemannian exponential map
\[
\exp_g\colon TM \to M
\]
is defined on all of $TM$ as $M$ is compact;
moreover, $\exp_g$ is real analytic.
There exists an open neighbourhood $\Omega\sub TM$
of the zero-section such that
\[
\alpha:=\exp_g|_\Omega\colon\Omega\to M
\]
is a real-analytic local addition in the sense that $\alpha(0_x)=x$
for all $x\in M$ and
\[
(\pi_{TM},\alpha) \colon \Omega\to M\times M
\]
has open image and is a real-analytic diffeomorphism onto its image,
where $\pi_{TM}\colon TM\to M$ is the bundle projection.\footnote{This is well known.
As $\exp_g(0_x)=x$ for $0_x\in T_xM$
and $T_0(\exp_g|_{T_xM})=\id_{T_xM}$
for all $x\in M$, it follows, e.g., from
the inverse function theorem
for $C^\omega$-maps
and \cite[Lemma~4.6]{DaG}.}
We let~$M^*$ be a complexification of~$M$ such that $M\sub M^*$
(see \cite{WaB} or \cite{Gra}).
After shrinking~$M^*$ if necessary we may assume that
\[
M=\{z\in M^*\colon \sigma(z)=z\}
\]
for an antiholomorphic map $\sigma\colon M^*\to M^*$
which is an involution (i.e., $\sigma\circ\sigma=\id_{M^*}$),
see \cite{WaB};
moreover, we may assume that~$M^*$
is $\sigma$-compact and hence metrizable.
Note that $T(M^*)$ is a complexification of~$TM$.
Hence~$\alpha$ extends to a complex-analytic map
\[
\alpha^*\colon \Omega^*\to M^*
\]
on some open subset $\Omega^*\sub T(M^*)$ with $\Omega\sub\Omega^*$
(see, e.g., \cite[Lemma~2.2\,(a)]{DGS}).
After shrinking~$\Omega^*$ and~$\Omega$, we may assume
that the map
\[
\theta:=(\pi_{T(M^*)},\alpha^*)\colon\Omega^*\to M^*\times M^*
\]
has open image and is a local $C^\infty_\C$-diffeomorphism
(exploiting the inverse function theorem),
and in fact a $C^\infty_\C$-diffeomorphism
onto its open image (using \cite[Lemma~4.6]{DaG}).
After replacing $\Omega^*$ with $\Omega^*\cap T\sigma(\Omega^*)$,
we may assume that $T\sigma(\Omega^*)=\Omega^*$.
After replacing~$\Omega^*$ with the union of its
connected components~$C$ which intersect~$TM$
non-trivially, we may assume that each~$C$ does so.
Since
$\sigma\circ \alpha^*\circ T\sigma|_{\Omega^*}$ and~$\alpha^*$
are complex analytic maps which coincide on $\Omega^*\cap TM$,
using the Identity Theorem we deduce that
\[
\sigma\circ\alpha^*\circ T\sigma|_{\Omega^*}=\alpha^*.
\]
\end{numba}
\begin{numba}
Let $\cU$ be the set of all
open neighbourhoods~$U$ of~$M$ in~$M^*$
such that $U=\sigma(U)$ and each connected component
of~$U$ meets~$M$.
We endow the complex vector space
$\Gamma^\infty_\C(TU)$
of complex-analytic vector fields on~$U$
with the compact-open topology, which turns it into a Fr\'{e}chet space
and coincides with the compact-open $C^\infty$-topology
(see \ref{vftopo}).
We endow the space
\[
\Germ(M,T(M^*))=\dl\, \Gamma^\infty_\C(TU)\vspace{-1.1mm}
\]
of germs of complex-analytic vector fields on open neighbourhoods $U\in \cU$
of~$M$ in~$M^*$
with the locally convex direct limit topology.
Each real-analytic vector field~$X$ on $M$ extends to
a complex-analytic vector field~$X^*\in \Gamma^\infty_\C(TU)$
for some~$U\in\cU$,
whose germ~$[X^*]$ around~$M$ is uniquely determined by~$X$.
In this way, we obtain an injective real linear map
\[
\Gamma^\omega(TM)  \to\Germ(M,T(M^*)),\quad X\mto [X^*];
\]
we endow the space $\Gamma^\omega(TM)$ of real-analytic vector fields
on~$M$ with the initial topology with respect to this map.
Whenever convenient, we shall identify $X\in \Gamma^\omega(TM)$
with~$[X^*]$. Thus, we consider $\Gamma^\omega(TM)$
as a vector subspace of $\Germ(M,T(M^*))$, endowed with the induced
topology.
Using this identification, we have that
\begin{equation}\label{nxtiscx}
\Gamma^\omega(TM)  =\{[X]\in \Germ(M,T(M^*))\colon [(T\sigma)\circ X\circ\sigma]=[X]\}.
\end{equation}
The mappings $\Gamma^\infty_\C(TU)\to \Gamma^\infty_\C(TU)$,
$X\mto T\sigma\circ X\circ \sigma$
are antilinear involutions for each open set $U\in\cU$.
Via the universal property of the locally convex direct limit,
they induce a continuous antilinear map
\[
\Germ(M,T(M^*))\to \Germ(M,T(M^*)),\quad
[X]\mto [T\sigma\circ X\circ \sigma]
\]
which is an involution. We now deduce from (\ref{nxtiscx}) that
\[
\Germ(M,T(M^*))=\Gamma^\omega(TM)_\C
\]
and, for $U$ in the directed set~$\cU$,
\[
\Gamma^\omega(TM)=\dl\;\Gamma^\infty_\C(TU)_\R\quad\mbox{with}
\]
$\Gamma^\infty_\C(TU)_\R:=
\{X\!\in\! \Gamma^\infty_\C(TU)\colon T\sigma\circ X\circ \sigma\!=X\}
\!=\!\{X\!\in\! \Gamma^\infty_\C(TU)\colon \! X(M)\!\sub \!TM\}$.
\end{numba}
\begin{numba}\label{def-funda}
We call a sequence $U_1\supseteq U_2\supseteq\cdots$ in~$\cU$
a \emph{fundamental sequence}
if, for each $U\in\cU$,
there exists $n\in\N$ such that $U_n \sub U$.
Then $\{U_n\colon n\in\N\}$ is cofinal in $(\cU,\supseteq)$
and thus
\begin{equation}\label{also-regu}
\Gamma^\omega(TM)=\dl\;\Gamma^\infty_\C(TU_n)_\R.
\end{equation}
\end{numba}
\begin{numba}\label{roter-faden}
In the following, we construct a fundamental sequence
$(U_n)_{n\in\N}$ (of the form $U_n=A_{\ve_n}$
with notation as below)
and real Banach spaces~$E_n$ such that
\[
\Gamma^\infty_\C(TU_1)_\R\;\sub \;E_1\;
\sub \; \Gamma^\infty_\C(TU_2)_\R \;\sub \;E_2\; \sub\; \cdots
\]
is a direct sequence of locally convex spaces
and the inclusion maps $E_n\to E_{n+1}$
are compact operators.
As the direct limit of the latter sequence coincides with that
of even and odd terms, respectively, we find that
\[
\Gamma^\omega(TM)=\dl\,\Gamma^\infty_\C(TU_n)_\R=\dl\,E_n,\vspace{-1.1mm}
\]
which implies the known fact that $\Gamma^\omega(TM)$
is a Silva space.
\end{numba}
\begin{numba}
The map $M^*\to T(M^*)$ taking $x\in M^*$ to
the zero-element $0_x\in T_x(M^*)$ is continuous,
whence
\[
O:=\{x\in M^*\colon 0_x\in \Omega^*\}
\]
is open in~$M^*$. Since $\Omega\sub\Omega^*$,
we have $M\sub O$.
\end{numba}
\begin{numba}
For each $x\in M$,
there exists a chart $\phi_x\colon Y_x\to Z_x\sub\C^m$
of $M^*$ with~$x\in Y_x$
such that $\phi_x(Y_x \cap M)=Z_x\cap \R^m$.
After shrinking~$Y_x$ and~$Z_x$,
we may assume that $Y_x\sub O$.
After replacing~$\phi_x$ with $\phi_x-\phi_x(x)$, we may assume that $\phi_x(x)=0$.
After replacing $Y_x$ with $Y_x\cap \sigma(Y_x)$,
we may assume that $Y_x=\sigma(Y_x)$; after passage
to the union of all connected components~$C$ of~$Y_x$ which intersect~$M$
non-trivially, we may assume that any~$C$
has this property.
Recall that
$\tau\colon\C^m\to\C^m$ is the complex conjugation.
As both of the complex-analytic functions
$\tau\circ\phi_x\circ \sigma|_{Y_x}$ and $\phi_x$
coincide on $M\cap Y_x$, we deduce that
\begin{equation}\label{good-phi}
\tau\circ\phi_x\circ\sigma|_{Y_x}=\phi_x.
\end{equation}
Notably, $Z_x=\tau(Z_x)$.
For some $a>0$, we have $[{-a},a]^m\sub Z_x$;
after replacing $\phi_x$ with $\frac{2}{a}\phi_x$,
we may assume that
\[
[{-2},2]^m\sub Z_x.
\]
After shrinking~$Z_x$ and~$Y_x$,
we may assume that $Z_x=W_{r_x}$ (as in~\ref{convents})
for some $r_x>0$.
By compactness of~$M$, there are $\ell\in\N$
and $x_1,\ldots,x_\ell\in M$ with
\[
M=\bigcup_{k=1}^\ell\phi_{x_k}^{-1}(]{-1},1[^m).
\]
We now simply write $\phi_k:=\phi_{x_k}$,
$Y_k:=Y_{x_k}$, and $Z_k:=Z_{x_k}$
for $k\in\{1,\ldots,\ell\}$.
Let $r:=\min\{r_1,\ldots, r_\ell\}$.
After shrinking~$Z_k$ and $Y_k$,
we may assume that
\[
Z_k=W_r\quad\mbox{for all $\, k\in\{1,\ldots,\ell\}$.}
\]
\end{numba}
\begin{numba}\label{new-num}
For all $k\in\{1,\ldots,\ell\}$,
we have $0_x\in \Omega^*$ for all $x\in Y_k$
(as $Y_k\sub O$) and
\begin{equation}\label{saves-a}
\alpha^*(0_x)=x\;\;\mbox{for all $\, x\in Y_k$.}
\end{equation}
To see this, note that the map
$W_r\to M^*$,\vspace{-.8mm} $y\mto \alpha^*(T\phi_k^{-1}(y,0))=\alpha^*(0_{\phi_k^{-1}(y)})$
is complex analytic;
we show that it coincides with the complex-analytic map~$\phi_k^{-1}$
on $W_r\cap\R^m$. The two maps then coincide
by the Identity Theorem, and the assertion follows.
Let $y\in W_r\cap\R^m$. Then $x:=\phi_k^{-1}(y)\in M$
and thus $\alpha^*(0_x)=\alpha(0_x)=x=\phi_k^{-1}(y)$ indeed.
\end{numba}
\begin{numba}\label{inverse-process}
Let $\Omega_k:=T\phi_k((\alpha^*)^{-1}(Y_k)\cap TY_k)
\sub TW_r=W_r \times\C^m$
and define
\[
\alpha_k\colon \Omega_k\to W_r,\quad (z,w)\mto \phi_k(\alpha^*(T(\phi_k^{-1})(z,w))).
\]
Then $W_r\times\{0\}\sub \Omega_k$,
as a consequence of~(\ref{saves-a}).
Note that
\[
\frac{d}{dt}\Big|_{t=0}\alpha(tv)=\frac{d}{dt}\Big|_{t=0}\exp_g(tv)=v
\]
for all $x\in M$ and $v\in T_xM$,
entailing that $\frac{d}{dt}\big|_{t=0}\alpha^*(tv)=v$
for all $x\in M$ and $v\in T_x(M^*)\cong (T_xM)_\C$.
As a consequence, the map $y\mto (\alpha_k)_x(y):=\alpha_k(x,y)$ has derivative
\[
(\alpha_k)_x'(0)=\id_{\C^{m}}
\]
for all $x\in W_r\cap\R^m$ and hence
for all $x\in W_r$, by the Identity Theorem.
Define
\[
\theta_k\colon \Omega_k\to W_r\times W_r,\quad
(z,w)\mto (z,\alpha_k(z,w)).
\]
Applying Lemma~\ref{final-loc}
to~$\alpha_k$ and~$\Omega_k$
in place of~$\alpha$ and~$\Omega$,
with $r/2$ in place of~$\ve$,
we find $\delta_k>0$
such that $W_{r/2}\times B_{\delta_k}^{\C^m}(0)\sub \Omega_k$
and, for each open subset $U\sub W_{r/2}$
and $\delta\in\;]0,\delta_k]$, the open set
$\theta_k(U\times B_\delta^{\C^m}\!(0))$ contains
the open subset
\[
\bigcup_{z\in U}(\{z\}\times B_{\delta/2}^{\C^m}(z))
\quad\mbox{of $\,W_r\times W_r$}\vspace{-1.1mm}
\]
and $\pr_2\circ \, \theta_k(z,\cdot)^{-1}|_{B^{\C^m}_{\delta/2}(z)}$
is Lipschitz with Lipschitz constant~$2$,
for each $z\in U$.
We set $\delta_0:=\min\{\delta_1,\ldots,\delta_\ell\}$.
Using notation as in \ref{convents},
for each $\delta\in\;]0,\delta_0]$, $\ve\in \;]0,r]$
and $k\in\{1,\ldots, \ell\}$ we now have
\[
\theta_k(V_\ve\times B_\delta^{\C^m}\!(0))
\; \supseteq\;
\bigcup_{z\in V_\ve}(\{z\}\times B_{\delta/2}^{\C^m}(z)).
\]
For $k\in\{1,\ldots, \ell\}$, we claim that
\[
\theta_k(\wb{z},\wb{w})=\wb{\theta_k(z,w)}\quad\mbox{for all $(z,w)\in V_r\times
B^{\C^m}_{\delta_0}(0)$;}
\]
if this is true, then
\begin{equation}\label{prereal}
\big(\forall (z,w)\in V_r\times B^{\C^m}_{\delta_0}(0)\big)\;\;
\theta_k(z,w)\in \R^m\times\R^m\;\Leftrightarrow
\;
(z,w)\in\R^m\times\R^m.
\end{equation}
Notably,
\begin{equation}\label{realdomim}
\pr_2(\theta_k^{-1}(x,z))\in\R^m\quad\mbox{for all $\,x\in V_r\cap\R^m$
and $\,z\in \R^m\cap B^{\C^m}_{\delta_0/2}(x)$,}
\end{equation}
which shall be useful later. To prove the claim,
let $x\in V_r\cap\R^m$. Then $T\phi_k^{-1}(x,0)=0_{\phi_k^{-1}(x)}\in\Omega$
holds.
By continuity of $\phi_k^{-1}$,\vspace{-.4mm} there exists an open\linebreak
neighbourhood
$A$ of $x$ in $V_r\cap\R^m$ and an open $0$-neighbourhood
$B$ in $\R^m\cap B_{\delta_0}^{\C^m}(0)$
such that $T\phi_k^{-1}(A\times B)\sub\Omega$.
Then $(z,w)\mto \theta_k(z,w)$ and
\[
(z,w)\mto\wb{\theta_k(\wb{z},\wb{w})}
\]
are complex-analytic functions $V_r\times B^{\C^m}_{\delta_0}(0)\to\C^m\times\C^m$
which coincide on $A\times B$ and hence on
$V_r\times B^{\C^m}_{\delta_0}(0)$, by the Identity Theorem.
\end{numba}
\begin{numba}
For $\ve\in\,]0,r/2]$ and $k\in\{1,\ldots,\ell\}$,
define $A_{k,\ve}:=\phi_k^{-1}(W_{2\ve})$,
$B_{k,\ve}:= \phi_k^{-1}(V_\ve)$,
$\phi_{k,\ve}:=\phi_k|_{A_{k,\ve}}\colon A_{k,\ve}\to W_{2\ve}$,
$\psi_{k,\ve}:=\phi_k|_{B_{k,\ve}}\colon B_{k,\ve}\to V_\ve$,
\[
A_\ve:=\bigcup_{k=1}^\ell A_{k,\ve},\quad\mbox{and}\quad
B_\ve:=\bigcup_{k=1}^\ell B_{k,\ve}.
\]
Thus $B_\ve\sub A_\ve$ for each $\ve$.
If $0<\ve'<\ve$, then $A_{\ve'}\sub A_\ve$ and $B_{\ve'}\sub B_\ve$.
We have $\sigma(A_{k,\ve})=\sigma(\phi_k^{-1}(W_{2\ve}))
=\phi_k^{-1}(\tau(W_{2\ve}))=\phi_k^{-1}(W_{2\ve})=A_{k,\ve}$.
Likewise, $\sigma(B_{k,\ve})=B_{k,\ve}$. As a consequence,
$A_\ve, B_\ve\in\cU$.
\end{numba}
\begin{numba}\label{hencefund}
For each $U\in\cU$, there exists
$\ve\in\;]0,r/2]$ such that $A_\ve\sub U$.
In fact, $Q_k:=\phi_k(U\cap Y_k)$ is an open subset
of $W_r$ which contains $[{-2},2]^m$.
The closure $\wb{W_{r/2}}$ is compact in $W_r$
and $(\wb{W_{r/2}}\setminus Q_k)\cap \bigcap_{\ve<r/2}\wb{W_{2\ve}}=\emptyset$
as $\bigcap_{\ve<r/2}\wb{W_{2\ve}}=[{-2},2]^m$.
Using the finite intersection property, we find
$\nu_k\in\; ]0,r/2]$ such that $\wb{W_{2\nu_k}}\sub Q_k$.
Let $\ve:=\min\{\nu_1,\ldots,\nu_\ell\}$.
Then $A_{k,\ve}=\phi_k^{-1}(W_{2\ve})\sub \phi_k^{-1}(Q_k)=U\cap Y_k\sub U$
for each $k\in\{1,\ldots,\ell\}$
and thus $A_\ve\sub U$.
\end{numba}
\begin{numba}\label{the-epsis}
We let $\ve_1:=r/2$.
Using~\ref{hencefund},
we can pick a strictly decreasing sequence
$(\ve_n)_{n\in\N}$ of positive real numbers such that $\ve_n\to 0$ as $n\to\infty$
and
\[
A_{\ve_{n+1}}\sub \, B_{\ve_n}\quad\mbox{for all $\, n\in \N$.}
\]
Then $(A_{\ve_n})_{n\in\N}$ is a fundamental
sequence in~$\cU$; we abbreviate $U_n:=A_{\ve_n}$.
\end{numba}
\begin{numba}\label{theEn}
For each $n\in\N$, we consider the map
\[
\Lambda_n\colon \Gamma^\infty_\C(TU_n)_\R\to\prod_{k=1}^\ell
\Hol(W_{2\ve_n},\C^m)_\R,\quad
X\mto (d\phi_{k,\ve_n}\circ X\circ \phi_{k,\ve_n}^{-1})_{k=1}^\ell
\]
taking a vector field to its family of local representatives
and
\[
\lambda_n
\colon \Gamma^\infty_\C(TB_{\ve_n})_\R\to\prod_{k=1}^\ell
\Hol(V_{\ve_n},\C^m)_\R,\quad
X\mto (d\psi_{k,\ve_n}\circ X\circ \psi_{k,\ve_n}^{-1})_{k=1}^\ell.
\]
These mappings are real linear and topological embeddings
(homeomorphisms onto the image).
%
% later: may give citation or cf.
%
Given $j,k\in\{1,\ldots,\ell\}$ and $\ve\in \;]0,r/2]$,
define
\[
W_{j,k,\ve}:=\phi_{j,\ve}(A_{j,\ve}\cap A_{k,\ve})
\]
and let $\phi_{j,k,\ve}\colon W_{k,j,\ve}\to W_{j,k,\ve}$,
$x\mto \phi_{j,\ve}(\phi_{k,\ve}^{-1}(x))$
be the transition map between the charts~$\phi_{k,\ve}$ and~$\phi_{j,\ve}$.
Likewise, define $V_{j,k,\ve}:=\psi_{j,\ve}(B_{j,\ve}\cap B_{k,\ve})$
and let
$\psi_{j,k,\ve}\colon V_{k,j,\ve}\to V_{j,k,\ve}$
be the transition map between the charts~$\psi_{k,\ve}$ and~$\psi_{j,\ve}$.\\[2.3mm]
The image of $\Lambda_n$ is the closed vector subspace $S_n$
consisting of all $(\zeta_1,\ldots,\zeta_\ell)$
in the product such that, for all $j,k\in\{1,\ldots,\ell\}$
and $x\in W_{k,j,\ve_n}$,
we have
\[
\zeta_j(\phi_{j,k,\ve_n}(x))
=d\phi_{j,k,\ve_n}(x,\zeta_k(x)),
\]
i.e., the vector fields corresponding to
$\zeta_k|_{W_{k,j,\ve_n}}$ and $\zeta_j|_{W_{j,k,\ve_n}}$
are $\phi_{j,k,\ve_n}$-related.
The image $H_n$ of $\lambda_n$ is closed; it contains
$(\zeta_1,\ldots,\zeta_\ell)\in
\prod_{k=1}^\ell
\Hol(V_{\ve_n},\C^m)_\R$~with
\begin{equation}
\zeta_j(\psi_{j,k,\ve_n}(x))=
d\psi_{j,k,\ve_n}(x,\zeta_k(x))
\mbox{ for all $j,k\in\{1,\ldots,\ell\}$
and $x\in V_{k,j,\ve_n}$.}\label{condiH}
\end{equation}
Finally, we let $R_n$ be the closed vector subspace of
$\prod_{k=1}^\ell \Hol_b(V_{\ve_n},\C^m)_\R$
consisting of all $(\zeta_1,\ldots,\zeta_\ell)$
therein such that~(\ref{condiH}) holds.
Then $R_n$ is a Banach space.
The inclusion map
\begin{equation}\label{cts-inclu}
j_n\colon \prod_{k=1}^\ell\Hol_b(V_{\ve_n},\C^m)_\R\to
\prod_{k=1}^\ell \Hol(V_{\ve_n},\C^m)_\R
\end{equation}
is continuous linear and takes $R_n$ into~$H_n$. We give the vector subspace
\[
E_n:=(\lambda_n|^{H_n})^{-1}(R_n)
\]
of $\Gamma^\infty_\C(TB_{\ve_n})_\R$
the Banach space structure which makes
$\lambda_n|_{E_n}\colon E_n\to R_n$ an isometric
isomorphism. Since $j_n|_{R_n}\colon R_n\to H_n$
is continuous, so is the inclusion map $i_n\colon E_n\to\Gamma^\infty_\C(TB_{\ve_n})_\R$,
and we can compose with the continuous linear restriction map~$r_n$ from the latter space to $\Gamma^\infty_\C(TU_{n+1})_\R$
to get a continuous linear injective map
\[
E_n\to \Gamma_\C^\infty(TU_{n+1})_\R.
\]
Using a restriction map in each component, we get
a continuous linear map
\[
\rho_n\colon \prod_{k=1}^\ell\Hol(W_{2\ve_n},\C^m)_\R\to
\prod_{k=1}^\ell\Hol_b(V_{\ve_n},\C^m)_\R
\]
which takes $S_n$ into~$R_n$.
Using restriction maps in each component,
we obtain a compact operator
\[
h\colon \prod_{k=1}^\ell \Hol_b(V_{\ve_n},\C^m)_\R
\to \prod_{k=1}^\ell\Hol_b(V_{\ve_{n+1}},\C^m)_\R,
\]
by Lemma~\ref{gives-compact}.
Then also $h|_{R_n}$ is a compact operator.
Hence
\[
\kappa:=\rho_{n+1}\circ\Lambda_{n+1}\circ r_n\circ i_n\circ (\lambda_n|_{E_n})^{-1}\colon
R_n\to R_{n+1}
\]
is a compact operator.
In fact, this map is $h|_{R_n}$, viewed as a map
to the closed vector subspace~$R_{n+1}$
of $\prod_{k=1}^\ell\Hol_b(V_{\ve_{n+1}},\C^m)_\R$.
Then also the bonding map
\[
E_n\to E_{n+1}
\]
is a compact operator,
as it equals $(\lambda_{n+1})^{-1}|_{R_{n+1}}\circ \kappa\circ
\lambda_n|_{E_n}$.
We have therefore achieved the situation
outlined in~\ref{roter-faden}.
\end{numba}
\section{Proof of Theorem~C}\label{sec-proofB}
To prove Theorem~C, we retain the notation from the preceding section.\linebreak
Notably, we shall use
$M$, $M^*$, $\alpha$, $\theta$, $r$, $\delta_0$,
$A_{k,\ve}$, $B_{k,\ve}$, $A_\ve$, $\phi_{k,\ve}$,
$\psi_{k,\ve}$, $V_{j,k,\ve}$, and $\psi_{j,k,\ve}$
as in Section~\ref{sec-gens}.
The sets~$V_\ve$ and $W_\ve$ are as in~\ref{convents}.
\begin{numba}
Let $\Diff(M)$ be the Lie group of all smooth
diffeomorphisms $M\to M$,
which is modelled on the Fr\'{e}chet space
$\Gamma(TM)$ of all smooth vector fields on~$M$.
Recall that there is an open $0$-neighbourhood $V_2\sub \Gamma(TM)$
such that
\[
X(M)\sub \Omega
\]
and $\alpha\circ X\in \Diff(M)$ for each $X\in V_2$,
and such that
\[
U_2:=\{\alpha\circ X\colon X\in V_2\}
\]
is an open identity neighbourhood in $\Diff(M)$
and the map
\[
\psi_2\colon V_2\to U_2,\quad X\mto \alpha\circ X
\]
is a $C^\infty$-diffeomorphism (cf.\ \cite{Mic}, \cite{Mil}, \cite{Ham},
\cite{KaM}, \cite[2.7]{DaS}).\footnote{No confusion with
$V_\ve$ will arise.} Thus
\[
\phi_2:=\psi_2^{-1}\colon U_2\to V_2
\]
is a chart for $\Diff(M)$ around $\id_M$.
\end{numba}
\begin{numba}
There is an open $0$-neighbourhood $V\sub \Gamma^\omega(TM)$
such that
\[
X(M)\sub \Omega
\]
and
$\alpha\circ X\in \Diff^\omega(M)$
for each $X\in V$, and such that
\[
U:=\{\alpha\circ X\colon X\in V\}
\]
is an open identity neighbourhood in $\Diff^\omega(M)$
and the map
\[
\psi\colon V\to U,\quad X\mto \alpha\circ X
\]
is a $C^\infty$-diffeomorphism (cf.\
Theorem 2.6 and Proposition 2.9 in \cite{DaS}). Thus
\[
\phi:=\psi^{-1}\colon U\to V
\]
is a chart for $\Diff^\omega(M)$
around $\id_M$.
The inclusion map $\iota\colon \Diff^\omega(M)\to \Diff(M)$
is a smooth group homomorphism (cf.\ \cite[Proposition 2.8]{DaS}).
Hence, after shrinking $U$ and~$V$, we may assume
that $U\sub U_2$ (and thus $V\sub V_2$).
Let $i\colon\Gamma^\omega(TM)\to\Gamma(TM)$
be the inclusion map, which is continuous linear.
Then
\begin{equation}\label{hence-situ}
i\circ \phi=\phi_2\circ \iota|_U\,.
\end{equation}
\end{numba}
\begin{numba}\label{via-eval}
Let $\cg$ and $\cg_2$ be the Lie algebra of $\Diff^\omega(M)$
and $\Diff(M)$, respectively. Abbreviate $e:=\id_M$.
For $x\in M$,
let $\delta_x\colon \Diff(M)\to M$, $\gamma\mto \gamma(x)$
be evaluation at~$x$ (which is a smooth map)
and $\ev_x:=\delta_x\circ \iota \colon \Diff^\omega(M)\to M$,
$\gamma\mto\gamma(x)$,
which is smooth as well.
Then
\[
\beta:=d\phi|_{\cg}\colon \cg\to \Gamma^\omega(TM)\quad\mbox{and}\quad
\beta_2:=d\phi_2|_{\cg_2}\colon \cg_2\to \Gamma(TM)
\]
are isomorphisms of topological vector spaces.
It is well known that
\begin{equation}\label{dercha2}
\beta_2=(T_e\delta_x)_{x\in M},\quad
v\mto (T_e\delta_x(v))_{x\in M};
\end{equation}
see, e.g., \cite[Lemma~C.4\,(d)]{MeR}
(cf.\ also \cite[Remark~A.13]{AGS}).
Then also
\begin{equation}\label{dercha}
\beta=(T_e\ev_x)_{x\in M}.
\end{equation}
In fact, (\ref{hence-situ}) implies that
\[
i\circ d\phi|_{\cg}=d\phi_2|_{\cg_2}\circ T_e\iota
\]
and thus $i\circ \beta=\beta_2\circ L(\iota)$.
For $v\in \cg$, this implies that
$\beta(v)=i(\beta(v))=\beta_2(L(\iota)(v))=(T\delta_x T_e\iota(v))_{x\in M}
=(T(\delta_x\circ \iota)(v))_{x\in M}=(T\ev_x(v))_{x\in M}$.
\end{numba}
\begin{numba}
For $(\ve_n)_{n\in\N}$
as in \ref{the-epsis},
choose $\ve_n'\in \,]\ve_{n+1},\ve_n[$
for each $n\in\N$.
Let~$\Lambda_n$ be as in~\ref{theEn} and
\[
F_n:=\prod_{k=1}^\ell BC^1_\C(W_{2\ve_n'},\C^m)_\R.
\]
The map
\[
\Theta_n\colon \prod_{k=1}^\ell\Hol(W_{2\ve_n},\C^m)_\R
\to F_n
\]
which is the restriction map to $W_{2\ve_n'}$ in each component
is continuous linear.
Define a closed vector subspace
\[
D_n \sub \prod_{k=1}^\ell \Hol_b(V_{\ve_n'},\C^m)_\R
\]
and a Banach space structure isomorphic to~$D_n$
on a vector subspace $E(n)\sub \Gamma^\infty_\C(TB_{\ve_n'})_\R$
in analogy to $R_n\sub \prod_{k=1}^\ell \Hol_b(V_{\ve_n},\C^m)_\R$
and $E_n\sub \Gamma^\infty_\C(TB_{\ve_n})_\R$ in~\ref{theEn},
replacing~$\ve_n$ with~$\ve_n'$.
Thus, using $\zeta\!=\!(\zeta_1,\ldots,\zeta_\ell)\!\in\!\prod_{k=1}^\ell
\Hol_b(V_{\ve_n'},\C^m)_\R$,
\begin{equation}
D_n=\{\zeta\colon \!(\forall j,k)\;
\mbox{$(\id,\zeta_k)|_{V_{k,j,\ve_n'}}\!\!$
and $(\id,\zeta_j)|_{V_{j,k,\ve_n'}}\!\!$ are $\psi_{j,k,\ve_n'}$-related}\}.
\end{equation}
Let $\Xi_n\colon E(n) \to
\Gamma^\omega(TM)$ be the injective
continuous linear map taking~$X$ to~$X|_M$.
By construction, $\mu_n\colon E(n)\to D_n$, $X\mto (d\psi_{k,\ve_n'}\circ X\circ
\psi_{k,\ve_n'}^{-1})_{k=1}^\ell$
is an isomorphism of topological vector spaces.
\end{numba}
\begin{numba}
By (\ref{L1-is-DL}) in~\ref{muchi}
and \ref{roter-faden},
we have
\[
L^1([0,1],\Gamma^\omega(TM))=
\dl\, L^1([0,1],\Gamma^\infty_\C(TU_n)_\R).\vspace{-.7mm}
\]
Let $Q(n)$ be the set of all $\gamma\in Q_{\ve_n'}$
(as in \ref{thepq}) such that, moreover,
\[
\|\gamma\|_{\cL^1,\|\cdot\|_\infty}<\delta_0/4.\vspace{-.7mm}
\]
Let $P(n):=\{[\gamma]\colon\gamma\in Q(n)\}\sub P_{\ve_n'}$
and consider the direct product
$P(n)^\ell$ as an open convex $0$-neighbourhood
in
\[
L^1\big([0,1],(BC^1_\C(W_{2\ve_n'},\C^m)_\R)^\ell\big)\cong
L^1([0,1],BC^1_\C(W_{2\ve_n'},\C^m)_\R)^\ell;
\]
the identification with the product
will be reused, without mention.
Then
\[
O_n:=L^1([0,1],\Theta_n\circ\Lambda_n)^{-1}(P(n)^\ell)
\]
is a convex, open $0$-neighbourhood in
$L^1([0,1],\Gamma^\infty_\C(TU_n)_\R)$
and
\[
O_1\sub O_2\sub\cdots.
\]
\end{numba}
\begin{numba}
Given $[\gamma]\in O_n$, let us write
$[\Theta_n\circ \Lambda_n\circ \gamma]=([\gamma_1],\ldots,[\gamma_\ell])$
with $\gamma_1,\ldots,\gamma_\ell\in Q(n)$.
After replacing $\gamma(t)$ and $\gamma_1(t),\ldots,\gamma_\ell(t)$
with $0$ for $t$ in a Borel subset of $[0,1]$
of measure~$0$, we may assume that
\[
\Theta_n\circ\Lambda_n\circ\gamma=(\gamma_1,\ldots,\gamma_\ell).
\]
Thus $\gamma_k$ is the local representative of~$\gamma$ in the chart
$\phi_{k,\ve_n'}$, for all $k\in\{1,\ldots,\ell\}$
(i.e., $\gamma(t)|_{A_{k,\ve_n'}}$ and $\gamma_k(t)$
are $\phi_{k,\ve_n'}$-related for all $t\in[0,1]$).
The differential equation
\begin{equation}\label{onM}
\dot{y}(t)=\gamma(t)(y(t))
\end{equation}
on $A_{\ve_n'}$ satisfies local existence and
local uniqueness of Carath\'{e}odory solutions
as its version in the charts $\phi_{k,\ve_n'}$,
\[
y'(t)=\gamma_k(t)(y(t)),
\]
does so by Lemma~\ref{local-lem}.
Since
\[
\|\Psi^{[\gamma_k]}_{\ve_n'}(t)(y_0)-y_0\|
=\|\Psi^{[\gamma_k]}_{\ve_n'}(t)(y_0)-
\Psi^{[0]}_{\ve_n'}(t)(y_0)\|
< \delta_0/2
\]
for all $y_0\in V_{\ve_n'}$ by
definition of~$Q(n)$ and the Lipschitz
estimate in Lemma~\ref{local-lem}\,(b),
we can form
\begin{equation}\label{the-eta-k}
\zeta_{\gamma,k}(t)(y_0):=\pr_2\big(\theta_k^{-1}\big(y_0,\Psi^{[\gamma_k]}_{\ve_n'}(t)(y_0)
\big)\big)
\in B^{\C^m}_{\delta_0}(0),
\end{equation}
see~\ref{inverse-process}.
Using
\begin{equation}\label{fnk}
f_{n,k}\colon \!\!\bigcup_{y_0\in V_{\ve_n'}}\!\!
\{y_0\}\times B^{\C^m}_{\delta_0/2}(y_0)
\to B_{\delta_0}^{\C^m}(0)\sub\C^m,\;\,
(y_0,z)\mto \pr_2(\theta_k^{-1}(y_0,z))\vspace{-.7mm}
\end{equation}
(as discussed in Lemma~\ref{final-loc})
and the associated Lipschitz continuous mapping
\[
(f_{n,k})_*\colon D_{(f_{k,n})_*}\to BC(V_{\ve_n'},\C^m)
\]
on $D_{(f_{n,k})_*}\sub BC(V_{\ve_n'},\C^m)$ with Lipschitz constant~$2$
(as in Lemma~\ref{fstar-lip}), we can rewrite (\ref{the-eta-k}) as
\[
\zeta_{\gamma,k}(t)=(f_{n,k})_*\big(\Psi^{[\gamma_k]}_{\ve_n'}(t)\big).
\]
Thus
$\zeta_{\gamma,k}=C\big([0,1],(f_{n,k})_*)\big(\Psi^{[\gamma_k]}_{\ve_n'}\big)\big)$,
where
\[
C([0,1],(f_{n,k})_*)\colon C([0,1],D_{(f_{n,k})_*})\to
C([0,1],BC(V_{\ve_n'},\C^m)),\;\,
\zeta\mto (f_{n,k})_*\circ\zeta
\]
is Lipschitz continuous with Lipschitz constant~$2$
on the subset $C([0,1],D_{(f_{n,k})_*})$ of $C([0,1],BC(V_{\ve_n'},\C^m))$,
as a special case of Lemma~\ref{fstar-lip}.\\[2.3mm]
Since $L^1([0,1],\Theta_n\circ\Lambda_n)$
is continuous linear and thus Lipschitz continuous,
we deduce that $\zeta_{\gamma,k}\in C([0,1],BC(V_{\ve_n'},\C^m))$
is Lipschitz continuous in $[\gamma]\in O_n$.\\[2.3mm]
Since $\Psi^{[\gamma_k]}_{\ve_n'}(t)$ and $f_{n,k}$
are complex analytic, also the function\vspace{-.3mm} $(f_{n,k})_*\big(\Psi^{[\gamma_k]}_{\ve_n'}(t)\big)$
is complex analytic. Since $\Psi^{[\gamma_k]}_{\ve_n'}(t)$
maps $V_{\ve_n'}\cap\R^m$ into~$\R^m$\vspace{-.3mm}
and $f_{n,k}$ satisfies~(\ref{realdomim}),
also $(f_{n,k})_*\big(\Psi^{[\gamma_k]}_{\ve_n'}(t)\big)$
maps $V_{\ve_n'}\cap\R^m$ into~$\R^m$;
thus
\[
(f_{n,k})_*\big(\Psi^{[\gamma_k]}_{\ve_n'}(t)\big)\in\Hol_b(V_{\ve_n'},\C^m)_\R.
\]
By the preceding, $\zeta_\gamma\colon [0,1]\to \prod_{k=1}^\ell\Hol_b(V_{\ve_n'},\C^m)_\R$,
\[
t\mto (\zeta_{\gamma,1}(t),\ldots,\zeta_{\gamma,\ell}(t))
\]
is a continuous function.
\end{numba}
\begin{numba}
Let $\Phi^\gamma$
be the maximal flow of the differential
equation~(\ref{onM}) on~$A_{\ve_n'}\sub M^*$
(denoted $\Fl^\gamma$ in \cite{GaH}),
as in~\cite[10.23 and Definition 4.7]{GaH}.
Then
\begin{equation}\label{flo-via-loc}
\Phi^\gamma_{t,0}(x)
=
\phi_{k,\ve_n'}^{-1}(\Psi^{[\gamma_k]}_{\ve_n'}(t)(\phi_{k,\ve_n'}(x)))
\end{equation}
for all $x\in B_{k,\ve_n'}$,
as $\gamma(t)|_{A_{k,\ve_n'}}$ and $\gamma_k(t)$
are $\phi_{k,\ve_n'}$-related for all $t\in[0,1]$.
As a consequence,
\[
\zeta_\gamma(t)\in D_n\quad\mbox{for all $\,t\in[0,1]$.}
\]
To see this, let $j,k\in\{1,\ldots,\ell\}$
and $y_0\in V_{k,j,\ve_n'}$.
Let $x:=\psi_{k,\ve_n'}^{-1}(y_0)$.
Then $y_1:=\psi_{j,k,\ve_n'}\in V_{j,k,\ve_n'}$
and $x=\psi_{j,\ve_n'}^{-1}(y_1)$.
By the preceding, we have
\[
\zeta_{\gamma,k}(t)(y_0)=f_{n,k}\big(y_0,\Psi^{[\gamma_k]}_{\ve_n'}(t)(y_0)\big)
=d\psi_{k,\ve_n'}\big(\theta^{-1}\big(x,\Phi^\gamma_{t,0}(x)\big)\big)
\]
and $\zeta_{\gamma,j}(t)(y_1)
=d\psi_{j,\ve_n'}\big(\theta^{-1}(x,\Phi^\gamma_{t,0}(x)\big)\big)$.
As a consequence,
$\zeta_{\gamma,j}(t)(y_1)=d\psi_{j,k,\ve_n'}(\zeta_{\gamma,k}(t)(y_0))$,
whence $\zeta_{\gamma,k}(t)|_{V_{k,j,\ve_n'}}$
and $\zeta_{\gamma,j}(t)|_{V_{j,k,\ve_n'}}$
are $\psi_{j,k,\ve_n'}$-related.
\end{numba}
\begin{numba}
Summing up, we have a function
\[
O_n\to C([0,1],D_n),
\quad [\gamma]\mto \zeta_\gamma
\]
which is Lipschitz continuous since
it is so as a map to
$C([0,1],BC(V_{\ve_n'},\C^m)^\ell)
\cong C([0,1],BC(V_{\ve_n'},\C^m)^\ell$,
as each component
$\zeta_{\gamma,k}$ is Lipschitz continuous in~$[\gamma]$.\\[2.3mm]
As a consequence, also the function
\[
f_n\colon O_n\to C([0,1],\Gamma^\omega(TM)), \quad
[\gamma]\mto \xi_\gamma:=\Xi_n\circ\mu_n^{-1}\circ
\zeta_\gamma
\]
is Lipschitz.
The functions $f_n$ are
compatible as $n$ increases (as a consequence of Lemma~\ref{local-lem}\,(e)),
and thus
\[
f\colon \bigcup_{n\in\N}O_n\to C([0,1],\Gamma^\omega(TM)),\quad
[\gamma]\mto\xi_\gamma\vspace{-.8mm}
\]
is continuous, by Theorem~D.
Then $f(L)\sub C([0,1],V)$ for some open $0$-neighbourhood
$L\sub \bigcup_{n\in\N}O_n$,
and now $C([0,1],\psi)\circ f|_L\colon L\to C([0,1],\Diff^\omega(M))$
is a continuous map.
\end{numba}
\begin{numba}
Henceforth, let us write $\Fl^\gamma$ for the maximal flow
of the differential equation
\[
\dot{y}(t)=\gamma(t)(y(t))
\]
on~$M$, for $[\gamma]\in O_n$.
Moreover, write $\eta_\gamma(t):=\psi(\xi_\gamma(t))\in \Diff^\omega(M)$.
If $y_0\in M$, there exists $k\in\{1,\ldots,\ell\}$
such that $y_0\in B_{k,\ve_n'}$.
Since $\Psi^{[\gamma_k]}_{\ve_n'}(t)\in \Hol_b(V_{\ve_n'},\C^m)_\R$,
we have
$\Phi^\gamma_{t,0}(y_0)\in M$ for each
$t\in [0,1]$, by~(\ref{flo-via-loc}).
Thus $\Fl^\gamma_{t,0}(y_0)$
exists for all $t\in [0,1]$
and is given by
\[
\Fl^\gamma_{t,0}(y_0)=\Phi^\gamma_{t,0}(y_0)=\eta_\gamma(t)(y_0).
\]
Thus $\Fl^\gamma_{t,0}=\eta_\gamma(t)$.
\end{numba}
By the preceding, for each $[\gamma]\in L$
we have
\[
(\psi\circ f)([\gamma])=(t\mto \Fl^\gamma_{t,0})
=\Evol^r(L^1([0,1],\beta_2^{-1})([\gamma]))
\in \AC([0,1],\Diff(M))
\] (cf.\
\cite[Proposition~11.4]{MeR}),
whence
\[
(\iota\circ \psi\circ f)([\gamma])=\Evol^r(L^1([0,1],\beta^{-1})([\gamma]))
\]
by Lemma~\ref{recog},
which applies because of~(\ref{hence-situ})
(see also Remarks~\ref{single-hom}
and~\ref{recog-left-right}).
Now \cite[Theorem 4.3.13]{Nik},
combined with Lemma~\ref{la-left-right},
shows that $\Diff^\omega(M)$ is $L^1$-semiregular.
By the preceding, $\Evol^r\colon L^1([0,1],\cg)\to
C([0,1],\Diff^\omega(M))$
is continuous on some $0$-neighbourhood in
$L^1([0,1],\cg)$.
Thus $\Diff^\omega(M)$ is $L^1$-regular
by Lemma~\ref{la-left-right}
and Theorem~E. $\,\square$
\section{Proofs for Theorems~A and~B}
To prove Theorem~A, let $p\in [1,\infty]$.
Since $G:=\Diff^\omega(M)$ is $L^1$-regular by Theorem~C,
it is $L^p$-regular.
Let $\Diff(M)$ be the Lie group of all smooth diffeomorphisms
of~$M$.
As the inclusion map $G\to\Diff(M)$ is
a smooth group homomorphism (cf.\ \cite[Proposition 2.8]{DaS}) and
the left action $\Diff(M)\times M\to M$,
$(\phi,x)\mto \phi(x)$ is smooth (cf., e.g., \cite[Lemma~1.19\,(a)
and Proposition~1.23]{AGS}),
also the left action
\[
\Lambda\colon G\times M\to M,\quad (\phi,x)\mto \phi(x)
\]
is smooth. Thus also the right action
\[
\sigma\colon M\times G\to M,\quad (x,\phi)\mto \phi^{-1}(x)
\]
is smooth.
We identify $\cg:=T_eG$ with $\Gamma^\omega(TM)$
by means of the isomorphism
\[
T_eG\to \Gamma^\omega(TM),\quad v\mto (T\ev_x(v))_{x\in M}
\]
discussed in~\ref{via-eval},
where $\ev_x\colon G\to M$, $\phi\mto\phi(x)$ is
the evaluation at $x\in M$ (and thus $\ev_x=\Lambda(\cdot,x)$).
Using this identification, the tangent map $T\ev_x=T\Lambda(\cdot,x)$
corresponds to the evaluation map
\[
\Gamma^\omega(TM)\to TM,\quad X\mto X(x).
\]
The left and right evolution maps
$L^p([0,1],\Gamma^\omega(TM))\to \AC_{L^p}([0,1],G)$
exist by $L^p$-regularity, and are related via
\[
\Evol(-[\gamma])=\Evol^r([\gamma])^{-1}
\]
for $[\gamma]\in L^p([0,1],\Gamma^\omega(TM))$,
using pointwise the inversion $j\colon G\to G$, $g\mto g^{-1}$,
which satisfies $T_ej=-\id_{\cg}$.
Now consider the time-dependent fundamental vector field
$(-\gamma)_\sharp\colon [0,1]\to\Gamma^\omega(TM)$,
$t \mto (-\gamma(t))_\sharp$
associated with $-\gamma$, given~by
\begin{eqnarray*}
(-\gamma(t))_\sharp(x)&:=&(T\sigma(x,\cdot))(-\gamma(t))
=T(\sigma(x,\cdot)\circ j)(\gamma(t))=T(\Lambda(\cdot,x))(\gamma(t))\\
&=&\gamma(t)(x)\quad\mbox{for $t\in[0,1]$ and $x\in M$.}
\end{eqnarray*}
Thus $(-\gamma)_\sharp=\gamma$.
By \cite[Theorem~10.19]{GaH}, the differential equation
$\dot{y}(t)=(-\gamma(t))_\sharp(y(t))=\gamma(t)(y(t))$
on~$M$
satisfies local existence and local uniqueness
of Carath\'{e}odory solutions, maximal solutions
are defined on all of $[0,1]$,
and the corresponding flow is given by
\begin{eqnarray*}
\Fl^\gamma_{t,t_0}(y_0)&=&
\Fl^{(-\gamma)_\sharp}_{t,t_0}(y_0)
=\sigma(y_0,\Evol(-[\gamma])(t_0)^{-1}\Evol(-[\gamma])(t)) \\
&=&\Evol(-[\gamma])(t)^{-1}(\Evol(-[\gamma])(t_0)(y_0)) \\
&=&\Evol^r([\gamma])(t)(\Evol^r([\gamma])(t_0)^{-1}(y_0))
\end{eqnarray*}
for $y_0\in M$ and $t,t_0\in [0,1]$.
Thus
\[
\Fl_{t,0}^\gamma(y_0)=\Evol^r([\gamma])(t)(y_0)
\]
for all $t\in[0,1]$ and $y_0\in M$, and thus
\[
\Fl_{t,0}^\gamma=\Evol^r([\gamma])(t)
\]
is an $\Diff^\omega(M)$-valued
$\AC_{L^p}$-function of $t\in [0,1]$
which depends smoothly on~$[\gamma]$.
Fix $t_0\in[0,1]$. The evaluation
\[
\AC_{L^p}([0,1],G)\to G,\;\,\theta\mto \theta(t_0)
\]
at $t_0$ being smooth (see \cite[Lemma~4.2.30]{Nik}),
we see that $h([\gamma]):=(\Fl_{t_0,0}^\gamma)^{-1}\in G$
depends smoothly on $[\gamma]\in L^p([0,1],\Gamma^\omega(TM))$.
For $g\in G$, let
$\rho_g\colon G\to G$, $\psi\mto\psi \circ g$
be the smooth right translation by~$g$
in the Lie group~$G$.
Consider the constant function
$C_g\in \AC_{L^p}([0,1],G)$ given by $C_g(t):=g$.
The map
\[
G\to \AC_{L^p}([0,1],G),\quad g\mto C_g
\]
is smooth, as the conclusions of \cite[Lemma~4.9]{MeR}
also hold (with $\K:=\R$, $r:=\infty$,
and $L^p$ in place of~$\cE$)
in the setting of~\cite{Nik},
for Lie groups modelled on sequentially complete locally convex spaces
(with analogous proof).
Since
\[
\Fl^\gamma_{t,t_0}=\Fl_{t,0}^\gamma \circ \, (\Fl^\gamma_{t_0,0})^{-1}
\]
for each $t\in[0,1]$, we see that
\[
\Fl^\gamma_{\cdot,t_0}=\rho_{h([\gamma])}\circ \Evol^r([\gamma])
=\Evol^r([\gamma])C_{h([\gamma])}
\]
is an element of $\AC_{L^p}([0,1],G)$
which depends smoothly on $[\gamma]$.
For $y_0\in M$, the evaluation
$\ev_{y_0}\colon \Diff^\omega(M)\to M$, $\psi\mto\psi(y_0)$
at~$y_0$ is a smooth map.
Using~\ref{chain-abs},
we deduce that
$\eta_{t_0,y_0}=\Fl^\gamma_{\cdot,t_0}(y_0)=\ev_{y_0}\circ \Fl^\gamma_{\cdot,t_0}$
is an $\AC_{L^p}$-function.
This completes the proof. $\,\square$\\[2.3mm]
{\bf Proof of Theorem~B.}
By \cite[Theorem~B]{MeR},
the Lie group $G :=  \Diff_c(M)$
is $L^1$-regular in the sense of \cite[Definition 5.16]{MeR},
where Borel measurability is used
instead of Lusin measurability.
By \cite[Lemma~1.41\,(d)]{MeR},
its modelling space $\Gamma_c(TM)$
of compactly supported smooth vector fields
has a certain
\emph{Fr\'{e}chet exhaustion property}
introduced in \cite[Definition 1.38]{MeR},
abbreviated~(FEP).
As a consequence, $G$
also is $L^1$-regular in the sense of~\cite{Nik}
we are using here
(cf.\ \cite[Remark~2.11]{Nik}).
The left action
%
% later give reference to box products paper
%
$\Lambda\colon G \times  M \to  M$,
$(\phi,x)\mto \phi(x)$
is smooth. Thus also the right action
$\sigma\colon M\times G\to M$,
$(x,\phi)\mto \phi^{-1}(x)$
is smooth.
Using \cite[Lemma~C.4]{MeR},
we identify $\cg:=T_eG$ with $\Gamma_c(TM)$
by means of the isomorphism
$T_eG\to \Gamma_c(TM)$,
$v\mto (T\delta_x(v))_{x\in M}$,
where $\delta_x\colon G\to M$, $\phi\mto\phi(x)$ is
the evaluation at $x\in M$ (and thus $\delta_x=\Lambda(\cdot,x)$).
Using this identification, $T\delta_x=T\Lambda(\cdot,x)$
corresponds to the evaluation map
$\Gamma_c(TM)\to TM$, $X\mto X(x)$.
We can now repeat the remainder of the proof of Theorem~A,
replacing the symbols $\ev$ and $\Gamma^\omega$
with $\delta$ and $\Gamma_c$, respectively. $\,\square$
\section{Proof of Theorem~F, and related results}
We first show that $f$ is continuous at each $x\in U$.
As in the preceding proof, we may assume
that $x\in U_1$ and $x=0$.
After replacing~$U_n$ with $U_n\cap (-U_n)$,
we may assume that the convex open $0$-neighbourhood~$U_n$
satisfies $U_n=-U_n$ and thus $[{-1},1]U_n\sub U_n$.
Let $\Sph:=\{z\in\C\colon |z|=1\}$ be the circle group.
For each $n\in\N$,
the map $m_n\colon\Sph\times E_n\to E_n$, $(z,y)\mto zy$
is continuous, whence
\[
Q_n:=\bigcap_{z\in \Sph}zU_n=\{y\in E_n\colon \Sph\times\{y\}\sub m_n^{-1}(U_n)\}\vspace{-1.1mm}
\]
is open, using the Wallace Theorem~\cite{Key} which allows
the compact set $\{y\}$ on the right-hand side
to be inflated to an open $y$-neighbourhood.
Since $Q_1\sub Q_2\sub\cdots$
by construction, after replacing~$U_n$ with~$Q_n$
we may assume that each~$U_n$
is
an open absolutely convex $0$-neighbourhood
in the complex locally convex space~$E_n$.
Let $q_n$ be the Minkowski functional of~$\frac{1}{3}U_n$.\\[2.3mm]
Let $v\in \frac{1}{3}U_n$
and $w\in B^{q_n}_1(0)=\frac{1}{3}U_n$.
Then $v+zw\in U_n$ for all $z\in\C$ with $|z|<2$.
The map
\[
h\colon \{z\in\C\colon |z|<2\}\to F,\quad z\mto f_n(v+zw)
\]
is complex analytic
and $df_n(v,w)=h'(0)$.
If~$p$ is a continuous seminorm on~$F$
and
\[
M_{n,p}:=\sup p(U_n)\in[0,\infty[
\]
for $n\in\N$, then
\[
p(df_n(v,w))=p(h'(0))\leq M_{n,p}
\]
by the Cauchy estimates. As $df_n(v,\cdot)\colon E_n\to F$
is complex linear for $v\in \frac{1}{3}U_n$, we deduce that
\begin{equation}\label{estinicer}
p(f_n(v,w))\leq M_{n,p}\, q_n(w)\quad\mbox{for all $\,w\in E_n$.}
\end{equation}
For all $v,w\in \frac{1}{3}U_n$, by convexity $v+t(w-v)\in\frac{1}{3}U_n$ for all
$t\in [0,1]$ and
\begin{eqnarray*}
p(f(w)-f(v))&= &p(f_n(w)-f_n(v))\leq\int_0^1 p(df_n(v+t(w-v),w-v))\,dt\\
&\leq & M_{n,p}\,q_n(w-v),
\end{eqnarray*}
using the Mean Value Theorem (see \cite[Proposition 1.18]{Sme})
and (\ref{estinicer}).
Thus $f_n|_{\frac{1}{3}U_n}$ is Lipschitz continuous.
Hence~$f$ is continuous on $\frac{1}{3}U$, by Theorem~D.\\[2.3mm]
To establish complex analyticity of~$f$, let $\tilde{F}$ be a completion of~$F$
such that $F\sub\tilde{F}$.
Given $x\in U$ and $y\in E$, the set
\[
W:=\{z\in\C\colon x+zy\in U\}
\]
is open in~$\C$.
Consider $W\to \tilde{F}$, $z\mto f(x+zy)$.
There exists $n\in \N$ such that $x\in U_n$
and $y\in E_n$. Given $z_0\in W$,
after increasing~$n$ we may assume that
$x+z_0y\in U_n$. Then $x+zy\in U_n$ for
all~$z$ in an open neighbourhood $Z$ of~$z_0$ in~$\C$,
as $U_n$ is open in~$E_n$.
Moreover,
\[
\frac{d}{dz}f(x+zy)=df_n(x+zy,y)
\]
exists for all $z\in Z$.
In the terminology of~\cite{BaS},
we have shown that $f\colon U\to\tilde{F}$
is complex differentiable on $U\cap L$
for each affine line $L\sub E$.
Thus $f\colon U\to\tilde{F}$
is G-analytic in the sense of
\cite[Definition~5.5]{BaS}, by \cite[Proposition~5.5]{BaS}.
Being G-analytic and continuous,
the map $f\colon U\to \tilde{F}$ is complex analytic,
by \cite[Theorem~6.1\,(i)]{BaS}.
For each $x\in U$ and balanced open $0$-neighbourhood
$Y\sub E$ such that $x+Y\sub U$, we therefore
have
\[
f(x+y)=\sum_{k=0}^\infty\frac{1}{k!}\delta^k_xf(y),
\]
using the G\^{a}teaux derivatives
$\delta^k_xf\colon E\to\tilde{F}$, $y\mto\frac{d^k}{dz^k}\big|_{z=0}f(x+zy)$ 
which are $\tilde{F}$-valued continuous homogeneous polynomials.
For $x\in U$ and $y\in E$,
pick $n\in \N$ with $x\in U_n$ and $y\in E_n$;
then $x+zy\in U_n$ for $z\in\C$ close to~$0$
and
\[
\delta^k_{x+zy}f(y)=\delta^k_{x+zy}f_n(y)\in F
\]
for all $k\in\N_0$, by a simple induction.
Thus each $\delta^k_xf$ is an $F$-valued
continuous homogeneous
polynomial on~$E$, and thus $f\colon U\to F$
is complex analytic. $\,\square$
\begin{thm}\label{characteri}
Let $E=\bigcup_{n\in\N}E_n$ be the locally
convex direct limit of a direct sequence
$E_1\sub E_2\sub\cdots$ of complex locally convex spaces.
We assume that $E$ is Hausdorff and consider
a function $f\colon U\to F$
from an open subset $U\sub E$ to a complex
locally convex space~$F$.
Consider the following conditions:
\begin{itemize}
\item[\rm(a)]
$f$ is complex analytic.
\item[\rm(b)]
For each $x\in U$ and each continuous $\C$-linear map
$\Lambda\colon F\to Y$ to a complex Banach space~$Y$,
there exist $m\in \N$ with
$x\in E_m$ and open, convex $x$-neighbourhoods
$V_n\sub E_n$ for $n\geq m$ with
$V_m\sub V_{m+1}\sub \cdots$
such that $V_n\sub U$
for each $n\geq m$ and $\Lambda\circ f|_{V_n}$ is complex analytic and bounded.
\item[\rm(c)]
For each $x\in U$,
there exist $m\in \N$ with
$x\in E_m$ and open, convex $x$-neighbourhoods
$W_n\sub E_n$ for $n\geq m$ with
$W_m\sub W_{m+1}\sub \cdots$
such that $W_n\sub U$ for all $n\geq m$
and $f|_{W_n}$ is complex analytic.
\end{itemize}
Then {\rm(a)} holds if and only if both
{\rm(b)} and {\rm(c)} are satisfied.
If $F$ is Mackey complete,
then {\rm(a)} and {\rm(b)}
are equivalent.
\end{thm}
\begin{proof}
(a) implies~(c): If $x\in U$, there exist a convex, open $x$-neighbourhood
$W\sub U$ and $m\in\N$ such that $x\in E_m$. We can take $W_n:=W\cap E_n$
for $n\geq m$.\\[2.3mm]
(a) implies~(b): For $x\in U$ and $\Lambda$ as in~(b),
there exist a convex, open $x$-neighbourhood
$V\sub \{y\in U\colon \Lambda(f(y))\in B^Y_1(\Lambda(f(x)))\}$
and $m\in\N$ such that $x\in E_m$. Then $\Lambda\circ f|_V$
is bounded and we can take $V_n:=V\cap E_n$
for $n\geq m$.\\[2.3mm]
To complete the proof, let $\Gamma$ be the set of all continuous
seminorms~$p$ on~$F$. For any such, let $\tilde{F}_p$
be a completion of the normed space $F_p:=F/p^{-1}(\{0\})$
associated with~$p$ such that $F_p\sub \tilde{F}_p$
and $\pi_p\colon F\to \tilde{F}_p$
be the map sending $x\in F$ to its coset $x+p^{-1}(\{0\})$.
The topology on~$F$ is initial with respect to
the linear maps $\pi_p$ for $p\in\Gamma$.\\[2.3mm]
Now assume that (b) holds. Then $\tilde{\pi}_p\circ f$ is complex
analytic for each $p\in\Gamma$, by Theorem~F,
whence~$f$ is continuous. If, moreover, (c)
holds, for the open $x$-neighbourhood~$W$ we now see
as in the second half of the proof of Theorem~F
that $f|_W$ is G-analytic and thus complex analytic
(being continuous as just shown).
If~(b) holds and $F$ is Mackey complete,
we let $\tilde{F}$ be a completion of~$F$
such that $F\sub\tilde{F}$
and let $\tilde{\pi}_p \colon\tilde{F}\to\tilde{F}_p$
be the unique continuous extension of~$\pi_p$,
for $p\in\Gamma$. Then $\tilde{F}$,
together with the maps $\tilde{\pi}_p$,
is a projective limit of the spaces $\tilde{F}_p$,
using the continuous extensions of the maps
$x+p^{-1}(\{0\})\mto x+q^{-1}(\{0\})$
as the bonding maps $\tilde{F}_p\to\tilde{F}_q$
if $q\leq p$ pointwise.
Henceforth, we consider $f$ as a map to~$\tilde{F}$.
Since $\tilde{\pi}_p\circ f$
is complex analytic for each $p\in\Gamma$,
we deduce that $f$ is complex analytic as a map to~$\tilde{F}$
(cf.\ \cite[Lemma~10.3]{BGN}).
Since $f(U)\sub F$ and~$F$ is Mackey complete,
all G\^{a}teaux derivatives
$\delta^k_xf(y)$ have values
in~$F$, as they can be calculated using Cauchy's
integral formula for $k$th complex derivatives.
The continuous function $f\colon U\to F$ is therefore
locally given by series of continuous homogeneous
polynomials to~$F$, and thus complex analytic
as a map to~$F$.
\end{proof}
\begin{rem}
Condition~(c) in Theorem~\ref{characteri}
does not imply~(a). In fact,
if $E$ is a non-normable complex locally convex space
which is a direct limit $E=\dl\,E_n$\vspace{-.9mm}
of normed spaces $E_1\sub E_2\sub\cdots$ over~$\C$, then the
evaluation map
\[
f\colon E'\times E\to\C,\quad (\lambda, x)\mto\lambda(x)
\]
is not continuous (and hence not complex analytic)
if we use the topology of bounded
convergence on dual spaces (cf.\ \cite[p.\,2]{KaM}).
Yet, $E'\times E=\dl\, (E'\times E_n)$\vspace{-.9mm}
holds and the restriction $f|_{E'\times E_n}$ is continuous $\C$-bilinear
(and hence complex analytic) for each $n\in\N$.
\end{rem}
\appendix
\section{Calculus in locally convex spaces}\label{appA}
We record our conventions and notation
concerning $C^k$-maps and analytic maps
between open subsets of locally convex spaces,
and the corresponding manifolds.
\begin{numba}
If $E$ and $F$ are locally convex spaces over $\K\in\{\R,\C\}$
and $f\colon U\to F$ is an $F$-valued function
on an open subset $U\sub E$,
we write
\[
df(x,y):=D_yf(x):=\frac{d}{dt}f(x+ty):=\lim_{t\to 0}
\frac{1}{t}(f(x+ty)-f(x))
\]
for the directional derivative of~$f$ at $x\in U$
in the direction $y\in E$, if it exists
(with $0\not=t\in\K$ such that $x+ty\in U$).
Given $k\in\N_0\cup\{\infty\}$,
we say that a map $f\colon U\to F$
is $C^k_\K$ if $d^{\,(0)}f:=f$ is continuous,
the directional derivative
\[
d^{\,(j)}f(x,y_1,\ldots, y_j):=(D_{y_j}\cdots D_{y_1}f)(x)
\]
exists in~$F$ for all $j\in \N$ with $j\leq k$
and all $(x,y_1,\ldots,y_j)\in U\times E^j$,
and the mappings
$d^{\,(j)}f\colon U\times E^j\to F$ so obtained are continuous.
Then
\[
f^{\,(j)}(x)\colon E^j\to F,\quad (y_1,\ldots,y_j)\to d^{\,(j)}f(x,y_1,\ldots, y_j)
\]
is continuous and symmetric $j$-linear over~$\K$,
for each $x\in U$.
We also write~$C^k$ in place of $C^k_\R$.
The $C^\infty_\R$-maps are also called \emph{smooth}.
This approach to calculus in locally
convex spaces, which goes back to \cite{Bas},
is also known as Keller's $C^k_c$-theory~\cite{Kel}.
We refer to \cite{Res}, \cite{GaN},
\cite{Ham}, \cite{Mil}, and \cite{Nee} for introductions
to this approach to calculus in the case $\K=\R$;
both $\K=\R$ and $\K=\C$ are treated in \cite{GaN} and~\cite{BGN}.
For the corresponding concepts of manifolds
and Lie groups modelled on a locally convex space,
see~\cite{Res}, \cite{GaN}, \cite{Nee}, \cite{Sme}, and \cite{BGN}.
If~$M$ is a $C^k_\K$-manifold modelled on a locally
convex space with $k\in\N\cup\{\infty\}$,
we let $TM$ be its tangent bundle
and write $T_xM$ for the tangent space at $x\in M$.
If~$V$ is an open subset of a locally convex space~$E$ over~$\K$,
we identify $TV$ with $V\times E$, as usual.
If $f\colon M\to N$ is a $C^k_\K$-map between $C^k_\K$-manifolds with $k\geq 1$,
we write
$Tf\colon TM\to TN$ for its tangent map
and define $T^jM:=T(T^{j-1}M)$ and $T^jf:=T(T^{j-1}f)\colon T^jM\to T^jN$
for all $j\in\N$ such that $j\leq k$
(using $T^0M:=M$ and $T^0f:=f$).
In the case of a $C^k_\K$-map $f\colon M\to V\sub E$,
we write $df$ for the second component of the tangent map
$Tf\colon TM\to TV=V\times E$
(see \cite{GaN,Nee,BGN}).
\end{numba}
\begin{numba}\label{indu-Ck}
We recall that a map $f\colon E\supseteq U\to F$
as above is $C^{k+1}$ for $k\in \N$
if and only if $f$ is $C^1$ and
$df\colon U\times E\to F$ is~$C^k$
(cf.\ \cite[Lemma 1.14]{Res}).
\end{numba}
\begin{numba}\label{vftopo}
If $M$ and $N$ are $C^k_\K$-manifolds modelled
on locally convex spaces, we endow
the set $C^k_\K(M,N)$ of all $C^k_\K$-maps
from $M$ to $N$ with the \emph{compact-open $C^k$-topology},
which is initial with respect to the maps
\[
T^j\colon C^k(M,N)\to C(T^jM,T^jN),\quad f\mto T^jf
\]
for $j\in\N_0$ such that $j\leq k$, where
$C(T^jM,T^jN)$ is endowed with the compact-open
topology
(see \cite[Definition~I.5.1]{Nee} and
\cite[Definition 4.1.2]{GaN}).
If $M$ is a $C^{k+1}_\K$-manifold (with $\infty+1:=\infty$)
and $\pi_{TM}\colon TM\to M$ its tangent bundle,
we endow the space $\Gamma^k_\K(TM)$
of $C^k_\K$-sections of $\pi_{TM}$
(the space of $C^k_\K$-vector fields)
with the topology induced
by $C^k_\K(M,TM)$
(see \cite[Definition I.5.2]{Nee}
and \cite[Definition 4.1.24]{GaN}),
which makes it a locally convex space over~$\K$
(see \cite[Proposition 4.1.25]{GaN};
compare also \cite{Mic,Mil,Ham,KaM}).
If $\K=\R$
and $k=\infty$, we abbreviate $\Gamma(TM):=\Gamma^\infty(TM)$.
If $M$ is a $\sigma$-compact finite-dimensional
smooth manifold, then $\Gamma(TM)$
is a Fr\'{e}chet space (see, e.g., \cite[Proposition 4.1.28]{GaN};
cf.\ also \cite[Example 1.1.5]{Ham}).
Likewise, $\Gamma_\C^\infty(TM)$
is Fr\'{e}chet for each finite-dimensional,
$\sigma$-compact $C^\infty_\C$-manifold~$M$.
\end{numba}
\begin{numba}\label{pardiff}
Let $E_1$, $E_2$, and $F$ be
locally convex spaces over $\K\in\{\R,\C\}$.
Let $U_1\sub E_1$ and $U_2\sub E_2$ be open subsets and $f\colon
U_1\times U_2\to F$ be a $C^1_\K$-map.
Then
\[
df((x_1,x_2),(y_1,y_2))=d_1f(x_1,x_2,y_1)+d_2f(x_1,x_2,y_2)
\]
for all $(x_1,x_2)\in U_1\times U_2$ and $(y_1,y_2)\in E_1\times E_2$,
where
\[
d_1f(x_1,x_2,y_1) :=d(f(\cdot,x_2))(x_1,y_1)
\]
and $d_2f(x_1,x_2,y_2):=d(f(x_1,\cdot))(x_2,y_2)$,
see \cite[Proposition~1.2.8]{GaN}
or \cite[Proposition 1.20]{Sme}.
\end{numba}
See \cite[2.2]{GaH} for the following fact:
\begin{numba}\label{at-point}\label{pwchain}
Let $E$ and $F$ be locally convex spaces, $U\sub E$ be open and
$f \colon U \to F$ be a $C^1$-map. If
$I \sub \R$ is a non-degenerate interval,
$\eta\colon  I \to E$ a function with $\eta(I)\sub U$ and
$t_0\in I$ such that the derivative
$\eta'(t_0)$ exists, then also $(f\circ\eta)'(t_0)$
exists and $(f\circ\eta)'(t_0)=df(\eta(t_0),\eta'(t_0))$.
\end{numba}
\begin{numba}\label{dotty}
Let $M$ be a $C^1$-manifold modelled
on a locally convex space~$E$. Let $I\sub\R$ be a non-degenerate
interval, $\eta\colon I\to M$
be a continuous map and $t_0\in I$.
We say that \emph{$\eta$ is differentiable
at $t_0$} if $\phi\circ \eta\colon \eta^{-1}(U_\phi)\to V_\phi$
is differentiable at $t_0$ for some chart
$\phi\colon U_\phi\to V_\phi\sub E$ of~$M$
such that $\eta(t_0)\in U_\phi$.
By~\ref{pwchain},
the latter then holds for any such chart,
and the tangent vector
\begin{equation}\label{stasta}
\dot{\eta}(t_0):=T\phi^{-1}((\phi\circ\eta)(t_0),(\phi\circ\eta)'(t_0))
\in T_{\eta(t_0)}M
\end{equation}
is well defined, independent of the choice of~$\phi$.
\end{numba}
See \cite[2.5]{GaH} for the following fact:
\begin{numba}\label{chain-pw}\label{C1onpw}
Let $f\colon M\to N$ be a $C^1$-map between
$C^1$-manifolds modelled
on locally convex spaces. If $I\sub \R$
is a non-degenerate interval, $t_0\in I$
and a continuous map $\eta\colon \!I \!\to\! M$ is differentiable
at~$t_0$, then $f\circ \eta\colon \!I\! \to\! N$ is differentiable at~$t_0$
and
\begin{equation}\label{tangpw}
(f\circ\eta)^\cdot(t_0)=Tf(\dot{\eta}(t_0)).
\end{equation}
\end{numba}
\begin{numba}
Let $E$ and $F$ be complex locally convex spaces
and $k\in\N_0$.
A mapping $p\colon E\to F$ is called a \emph{continuous
homogeneous polynomial of degree~$k$}
if there exists a continuous $k$-linear map
$\beta\colon E^k\to F$ such that $p(x)=\beta(x,\ldots,x)$
for all $x\in E$, with $k$ entries~$x$ (if $k=0$, we mean $p(x)=\beta(0)$).
A function $f\colon U\to F$ on an open subset $U\sub E$
is called \emph{complex analytic} if $f$ is continuous
and, for each $x\in U$, there exist continuous homogeneous
polynomials $p_k\colon E\to F$ of degree~$k$
and an open neighbourhood $V\sub U$ such that
\[
f(y)=\sum_{k=0}^\infty p_k(y-x)
\]
in~$F$, pointwise for $y\in V$.
A function $f\colon U\to F$ as before
is complex analytic if and only if it is $C^\infty_\C$;
if $F$ is sequentially complete (or at least Mackey complete
in the sense of~\cite{KaM}),
then $f$ is complex analytic if and only
if $f$ is $C^1_\C$
(see \cite[Propositions~7.4 and 7.7]{BGN}
or \cite[Theorem 2.1.12]{GaN}).
\end{numba}
\begin{numba}
If $E$ and $F$ are real locally convex spaces,
a function $f\colon U\to F$ on an open subset $U\sub E$
is called \emph{real analytic}
if there exist an open subset $U^*\sub E_\C=E\oplus iE$
with $U\sub U^*$ and a complex analytic function
$f^*\colon U^*\to F_\C$ such that
$f^*|_U=f$ (see \cite[Definition 2.3]{Res} and \cite[Definition 2.2.2]{GaN},
or already \cite{Mil}
if $E$ and $F$ are sequentially complete).
\end{numba}
Since \cite[Lemma~1.6.28]{GaN}
is not yet publicly available,
we record two lemmas.
\begin{la}\label{bilnew}
Let $E_1$, $E_2$, and $F$ be real vector spaces,
$\beta\colon E_1\times E_2\to F$
be a bilinear map and $q$, $q_1$, and $q$ be seminorms on $E_1$, $E_2$,
and $F$, respectively, such that
$\beta(x,y)\in B^q_1(0)$
for all $x\in B^{q_1}_1(0)$
and $y\in B^{q_2}_1(0)$.
Then
\[
q(\beta(x,y))\leq q_1(x)q_2(y)\quad\mbox{for all $(x,y)\in E_1\times E_2$.}
\]
\end{la}
\begin{proof}
Let $(x,y)\in E_1\times E_2$.
For all $s>q_1(x)$ and $t>q_2(y)$,
we have $q_1((1/s)x)<1$ and $q_2((1/t)y)<1$, whence
$q(\beta((1/s)x,(1/t)y))<1$ and hence
\[
q(\beta(x,y))<st.
\]
Letting $s\to q_1(x)$ and $t\to q_2(y)$, we deduce that
$q(\beta(x,y))\leq q_1(x)q_2(y)$.
\end{proof}
\begin{la}\label{new}
Let $E_1$, $E_2$ and $F$ be locally convex spaces,
$U_j\sub E_j$ be an open $0$-neighbourhood for $j\in \{1,2\}$
and $f\colon U_1\times U_2\to F$ be a $C^2$-function
such that $f(x,0)=0$ for all $x\in U_1$ and $f(0,y)=0$ for all
$y\in U_2$. Let $q$ be a continuous seminorm on~$F$.
Then there exist continuous seminorms $q_j$ on $E_j$ for $j\in \{1,2\}$
and convex open $0$-neighbourhoods $V_j\sub U_j$
such that
\[
q(f(x,y))\leq q_1(x)q_2(y)\quad\mbox{for all $(x,y)\in V_1\times V_2$.}
\]
\end{la}
\begin{proof}
Abbreviate $E:=E_1\times E_2$.
As the mapping $d^2f\colon U_1\times U_2 \times E\times E\to F$
is continuous and $d^2f((0,0),(0,0),(0,0))=0$,
the pre-image $(d^2f)^{-1}(B^q_1(0))$
is an open $0$-neighbourhood and hence contains
\[
V_1\times V_2\times
B^p_1(0)\times B^p_1(0)
\]
for certain convex open $0$-neighbourhoods $V_j\sub E_j$
and a continuous seminorm $p$ on~$E$.
Thus
\begin{equation}\label{wapply}
q(d^2f((x,y),v,w))\leq p(v)p(w)\quad\mbox{for all
$(x,y)\in V_1\times V_2$ and $v,w\in E$,}
\end{equation}
by Lemma~\ref{bilnew}.
After increasing $p$ if necessary, we may assume
that there exist continuous seminorms $q_j$ on $E_j$
for $j\in \{1,2\}$ such that
$p(x,y)=\max\{q_1(x),q_2(y)\}$ for all $(x,y)\in E$.
Since $f(0,y)=0$
for all $y\in U_2$, we have
\begin{equation}\label{diffnu}
df((0,y),(0,z))=0\quad\mbox{for all $y\in U_2$ and $z\in E_2$.}
\end{equation}
Let $(x,y)\in V_1\times V_2$.
Using $f(x,0)=0$ and the identity (\ref{diffnu}),
two applications of the Mean Value Theorem (see \cite[Proposition 1.18]{Sme})
show that
\begin{eqnarray}
f(x,y)&=& f(x,y)-f(x,0)=\int_0^1 df((x,ty),(0,y))\, dt\\
&=&\int_0^1 df((x,ty),(0,y))-df((0,ty),(0,y))\, dt\\
&=&\int_0^1\int_0^1 d^2f((sx,ty),(0,y),(x,0))\, ds\, dt\,.
\end{eqnarray}
Using (\ref{wapply}), we deduce that
$q(f(x,y))\leq \int_0^1\int_0^1q\big(d^2f((sx,ty),(0,y),(x,0))\big)\, ds\, dt$
$\leq \int_0^1\int_0^1 p(0,y)p(x,0)\,ds\,dt\leq
p(0,y)p(x,0)=q_1(x)q_2(y)$.
\end{proof}
\section{Absolutely continuous functions in locally\\
convex spaces and Carath\'{e}odory solutions}\label{secB}
We refer to \cite{FMP} and \cite{Nik} for background
on vector-valued $L^p$-spaces. For absolutely continuous functions,
see \cite{Nik} and \cite{GaH} (cf.\ also \cite{MeR});
for Carath\'{e}odory solutions
to differential equations, see \cite{GaH}
(cf.\ also \cite{Nik} and \cite{MeR}).
We closely follow~\cite{GaH} and recall some essentials.
\begin{numba}\label{def-lus}\label{close-to-borel}
Let $\tilde{\lambda}\colon \tilde{\cB}(I)\to[0,\infty]$
be Lebesgue measure on an interval $I\sub\R$.
A mapping $\gamma\colon I\to X$
to a topological space~$X$
is called \emph{Lusin measurable}
if there exists a sequence $(K_j)_{j\in\N}$
of compact subsets $K_j\sub I$ such that
\begin{itemize}
\item[(i)]
The restriction $\gamma|_{K_j}\colon K_j\to X$ is
continuous for each $j\in \N$;
\item[(ii)]
$\tilde{\lambda}(I\setminus \bigcup_{j\in \N}K_j)=0$.
\end{itemize}
Compare~\cite{FMP,Nik} and the references therein for
further information, also~\cite{Tho}.
If $X$ is second countable,
then a map $\gamma\colon I\to X$ is Lusin measurable
if and only if $\gamma$ is measurable as a
function
from $(I,\tilde{\cB}(I))$ to $(X,\cB(X))$,
where $\cB(X)$ is the $\sigma$-algebra of Borel sets
of~$X$ (compare, e.g., \cite[Lemma~4.1.8]{Nik}).
\end{numba}
\begin{numba}
If $E$ is a locally convex space,
$I\sub \R$ an interval
and $p\in [1,\infty]$,
we write $\cL^p(I,E)$
for the vector space
of all Lusin measurable mappings $\gamma\colon I\to E$
such that $\|\gamma\|_{\cL^p,q}:=\|q\circ\gamma\|_{\cL^p}<\infty$
for all continuous seminorms $q$ on~$E$.
The set $L^p(I,E)$
of equivalence classes $[\gamma]$
modulo functions vanishing
almost everywhere is a locally convex
space whose topology is determined by the seminorms
$\|\cdot\|_{L^p,q}$ given by
$\|[\gamma]\|_{L^p,q}:=\|\gamma\|_{\cL^p,q}$;
see \cite[Definition~3.3]{FMP};
cf.\ also \cite{Nik} (where $I$ is compact).
We mention that a Lusin measurable
function $\gamma\colon I\to E$
is in $\cL^\infty$ if and only if $\gamma(A)$
is bounded in~$E$ for some $A\in\wt{\cB}(I)$
such that $\wt{\lambda}(I\setminus A)=0$
(see \cite[Definition~2.2]{FMP});
this follows from  a Localization Lemma,
\cite[Lemma~2.1]{FMP},
due to Thomas~\cite{Tho}.
If $\gamma\colon I\to E$ is \emph{locally $\cL^p$}
in the sense that $\gamma|_{[a,b]}\in\cL^p([a,b],E)$
for all $a<b$ with $[a,b]\sub I$,
again we write $[\gamma]$ for the equivalence
class modulo functions
vanishing almost everywhere.
\end{numba}
\begin{numba}
If $E$ is sequentially complete,
we call $\eta\colon I\to E$
an $\AC_{L^p}$-function if $\eta$
is the primitive of some
$\gamma\colon I\to E$ which is locally $\cL^p$, i.e.,
\[
\eta(t)=\eta(t_0)+\int_{t_0}^t\gamma(s)\,ds
\]
for all $t\in I$ and some (and then each) $t_0\in I$,
using weak $E$-valued integrals with respect to Lebesgue measure.
Then $\eta':=[\gamma]$ is uniquely determined (cf.\ \cite[Lemma~2.28]{Nik}).
The $\AC_{L^1}$-functions are also
called \emph{absolutely continuous};
each $\AC_{L^p}$-function is absolutely continuous.
\end{numba}
\begin{numba}\label{charu}\label{chain-abs} (Chain Rule).
Let $E$ and $F$ be sequentially complete locally convex spaces,
$U\sub E$ be an open subset, $f\colon  U\to F$ a $C^1$-map
and $\eta\colon I\to E$ be an $AC_{L^p}$-function
such that $\eta(I)\sub U$.
Let $\gamma\colon I\to E$ be a locally $\cL^p$-function with $\eta'=[\gamma]$.
Then $f\circ\eta\colon I\to F$ is $\AC_{L^p}$ and
\[
(f\circ\eta)'=[t\mto df(\eta(t),\gamma(t))],
\]
by \cite[Lemma~3.7]{Nik} and its proof.
\end{numba}
\begin{numba}\label{numba-ode}
If $E$ is a sequentially complete locally convex
space, $W\sub\R\times E$ a subset,
$f\colon W\to E$ a function
and $(t_0,y_0)\in W$,
we call a function $\eta\colon I\to E$ on a non-degenerate interval $I\sub \R$
a \emph{Carath\'{e}odory solution} to the initial value problem
\begin{equation}\label{the-ivp}
y'(t)=f(t,y(t)),\quad y(t_0)=y_0
\end{equation}
if $\eta$ is absolutely continuous, $t_0\in I$ holds,
$(t,\eta(t))\in W$ for all $t\in I$, and the integral
equation
\begin{equation}\label{cara2}
\eta(t)=\eta(t_0)+\int_{t_0}^tf(s,\eta(s))\,ds
\end{equation}
is satisfied for all $t\in I$, which is equivalent to
the condition
\begin{equation}\label{cara3}
\eta'=[t\mto f(t,\eta(t))]\quad\mbox{and}\quad
\eta(t_0)=y_0.
\end{equation}
Carath\'{e}odory solutions to $y'(t)=f(t,y(t))$ are solutions
to initial value problems for some choice of $(t_0,y_0)\in W$.
\end{numba}
\begin{numba}
Let $M$ be a $C^1$-manifold modelled on a locally
convex space and $TM$ be its tangent bundle, with the bundle
projection $\pi_{TM}\colon TM\to M$.
If $\gamma\colon I\to TM$ is a Lusin measurable
function on an interval $I\sub\R$, we write $[\gamma]$
for the set of all Lusin measurable functions
$\eta\colon I\to TM$
such that $\pi_{TM}\circ\gamma=\pi_{TM}\circ\eta$
and $\gamma(t)=\eta(t)$ for almost all $t\in I$.
\end{numba}
\begin{numba}
Let $p\in [1,\infty]$ and $M$ be a $C^1$-manifold modelled on a sequentially
complete locally convex space~$E$.
For real numbers $a<b$, consider a continuous function
$\eta\colon [a,b]\to M$.
If $\eta([a,b])\sub U_\phi$ for some chart $\phi\colon U_\phi\to V_\phi\sub E$
of~$M$,
we say that $\eta$ is an \emph{$\AC_{L^p}$-map}
if $\phi\circ \eta\colon I\to E$ is so,
and let
\[
\dot{\eta}:=[t\mto T\phi^{-1}((\phi\circ\eta)(t),\gamma(t))]
\]
with $\gamma\in\cL^p([a,b],E)$ such that $(\phi\circ\eta)'=[\gamma]$.
By \ref{charu}, the $\AC_{L^p}$-property of~$\eta$
is independent of the choice of~$\phi$,
and so is~$\dot{\eta}$.
In the general case,
we call $\eta$ an \emph{$\AC_{L^p}$-map}
if $[a,b]$ can be subdivided
into subintervals $[t_{j-1},t_j]$
such that $\eta([t_{j-1},t_j])$
is contained in a chart domain
and $\eta|_{[t_{j-1},t_j]}$ is $AC_{L^p}$.
If $(\eta|_{[t_{j-1},t_j]})^{\cdot}=[\gamma_j]$,
we let $\dot{\eta}:=[\gamma]$ with $\gamma(t):=\gamma_j(t)$
if $t\in [t_{j-1},t_j[$ or $j$ is maximal and $t\in [t_{j-1}, t_j]$.
If $I\sub\R$ is an interval, we call a function $\eta\colon I\to M$
an \emph{$\AC_{L^p}$-map} if $\eta|_{[a,b]}$ is so for all $a<b$ such that
$[a,b]\sub I$. We define $\dot{\eta}=[\gamma]$ where
$\gamma$ is defined piecewise using representatives
of $(\eta|_{[a,b]})^{\cdot}$
for $[a,b]$ in a countable cover of~$I$.
The $\AC_{L^1}$-maps are also called \emph{absolutely continuous}.
\end{numba}
\begin{numba}\label{chainRNEW}
Let $f\colon M\to N$ be a $C^1$-map between
$C^1$-manifolds modelled on sequentially complete
locally convex spaces.
Let $I\sub \R$ be a non-degenerate interval and $\eta\colon I\to M$
be absolutely continuous. Let $\gamma\colon I\to TM$
be a Lusin measurable function such that $\pi_{TM}\circ\gamma=\eta$
and $\dot{\eta}=[\gamma]$.
Then $f\circ\eta\colon I\to N$ is absolutely
continuous and
\[
(f\circ \eta)^{\cdot}=[t\mto Tf(\gamma(t))],
\]
as a consequence of~\ref{charu}.
\end{numba}
\begin{numba}\label{def-ode-mfd-2}
If $M$ is a $C^1$-manifold modelled on a sequentially
complete locally convex space, $W\sub\R\times M$ a subset
and $f\colon W\to TM$ a function such that $f(t,y)\in T_yM$
for all $(t,y)\in W$, given $(t_0,y_0)\in W$
we call a function $\eta\colon I\to M$ on a non-degenerate interval $I\sub \R$
a \emph{Carath\'{e}odory solution} to the initial value problem
\begin{equation}\label{ivp-xyz}
\dot{y}(t)=f(t,y(t)),\quad y(t_0)=y_0
\end{equation}
if $\eta$ is absolutely continuous, $t_0\in I$ holds,
$(t,\eta(t))\in W$ for all $t\in I$,
\begin{equation}\label{cara13}
\dot{\eta}\, =\, [t\mto f(t,\eta(t))],\quad\mbox{and}\quad
\eta(t_0)=y_0.
\end{equation}
Solutions to the differential equation $\dot{y}(t)=f(t,y(t))$
are defined analogously.
\end{numba}
For terminology and basic facts concerning
local existence and local uniqueness of Carath\'{e}odory
solutions, see~\cite{GaH}.
Helge Gl\"{o}ckner, Universit\"{a}t Paderborn, Warburger Str.\ 100,
33098 Paderborn, Germany; glockner@math.uni-paderborn.de\vfill

\begin{thebibliography}{99}
%
\bibitem{CRO}
Agrachev, A. A. and Y. L. Sachkov,
``Control Theory from the Geometric Viewpoint,''
Springer, Berlin, 2004. 
%
%
\bibitem{AaS}
Alzaareer, H. and A. Schmeding,
\emph{Differentiable mappings on products with different degrees
of differentiability in the two factors},
Expo.\ Math.\ {\bf 33} (2015),
%No. 2,
184--222.
%
%
\bibitem{AGS}
Amiri, H., H. Gl\"{o}ckner, and A. Schmeding,
\emph{Lie groupoids of mappings taking values in a Lie groupoid},
Arch.\ Math., Brno {\bf 56} (2020),
% No. 5,
307--356. 
%
%
\bibitem{Bas}
Bastiani, A., \emph{Applications diff\'erentiables et vari\'et\'es diff\'erentiables 
de dimension infinie}, J. Anal.\ Math.\ \textbf{13} (1964), 1--114.
%
%
\bibitem{BGN}
Bertram, W., H. Gl\"{o}ckner, and
K.-H. Neeb,
\emph{Differential calculus over general base fields and rings},
Expo.\ Math.\ {\bf 22} (2004), 213--282.
%
%
\bibitem{BaS}
Bochnak, J. and J. Siciak,
\emph{Analytic functions in topological vector spaces},
Stud.\ Math.\ {\bf 39} (1971), 77--112.
%
%
\bibitem{Dah}
Dahmen, R.,
\emph{Analytic mappings between LB-spaces and applications
in infinite-dimensional Lie theory},
Math.\ Z. {\bf 266} (2010), 115--140.
%
%
\bibitem{DaG}
Dahmen, R. and H. Gl\"{o}ckner,
\emph{Bounded solutions of finite lifetime to differential equations in Banach spaces},
Acta Sci.\ Math.\ (Szeged) {\bf 81} (2015), 457--468.
%
%
\bibitem{DGS}
Dahmen, R., H. Gl\"{o}ckner,
and A. Schmeding,
\emph{Complexifications of infinite-dimensional manifolds and new constructions of infinite-dimensional Lie groups}, preprint,
arXiv:1410.6468.
%
%
\bibitem{DaS}
Dahmen, R. and A.
Schmeding,
\emph{The Lie group of real analytic diffeomorphisms
is not real analytic},
Stud.\ Math.\ {\bf 229} (2015), 141--172. 
%
% 
\bibitem{FMP}
Florencio, M.,
F. Mayoral und P. J. Pa\'{u}l,
\emph{Spaces of vector-valued integrable functions and localization of bounded subsets},
Math.\ Nachr.\ {\bf 174} (1995), 89--111.
%
%
\bibitem{Flo}
Floret, K.,
\emph{Lokalkonvexe Sequenzen mit kompakten Abbildungen},
J. Reine Angew.\ Math.\ {\bf 247} (1971), 155--195.
%
%
\bibitem{FaW}
Floret, K. and J. Wloka,
``Einf\"{u}hrung in die Theorie der lokalkonvexen R\"{a}ume'',
Springer, Berlin, 1968.
%
%
\bibitem{Res}
Gl\"{o}ckner, H.,
\emph{Infinite-dimensional Lie groups without completeness restrictions},
pp.\ 43--59 in: Strasburger, A. et al.\ (eds.),
``Geometry and Analysis on Finite- and Infinite-Dimensional Lie Groups,''
Banach Center Publications {\bf 55}, Warsaw, 2002.
%
%
\bibitem{GCX}
Gl\"{o}ckner, H.,
\emph{Lie group structures on quotient groups and universal
complexifications for infinite-dimensional Lie groups},
J. Funct.\ Anal.\ {\bf 194} (2002), 347--409.
%
%
\bibitem{IMP}
Gl\"{o}ckner, H.,
\emph{Implicit functions from topological vector spaces to Banach spaces},
Isr.\ J. Math.\ {\bf 155} (2006), 205--252.
%
%
\bibitem{SEM}
Gl\"{o}ckner, H.,
\emph{Regularity properties of infinite-dimensional Lie groups,
and semiregularity}, preprint, arXiv:1208.0715.
%
%
\bibitem{MeR}
Gl\"{o}ckner, H.,
\emph{Measurable regularity properties
of infinite-dimensional Lie groups},
preprint, arXiv:1601.02568.
%
%
\bibitem{GaH}
Gl\"{o}ckner, H. and J. Hilgert,
\emph{Aspects of control theory on infinite-dimensional Lie groups
and $G$-manifolds},
J. Differential Equations {\bf 343} (2023), 186--232.
%
%
\bibitem{GaN} Gl\"{o}ckner, H. and K.-H. Neeb,
``Infinite-Dimensional Lie Groups,'' book in preparation.
%
%
\bibitem{Gra}
Grauert, H.,
\emph{On Levi's problem and the imbedding of real-analytic manifolds},
Ann.\ Math.\ {\bf 68} (1958), 460--472.
%
% 
\bibitem{Ham}
Hamilton, R.\,S.,
\emph{The inverse function theorem of Nash and Moser},
Bull.\ Amer.\ Math.\ Soc.\ {\bf 7} (1982), 65--222.
%
%
\bibitem{Tro}
Hanusch, M.,
\emph{The strong Trotter property for locally $\mu$-convex Lie groups},
J. Lie Theory {\bf 30} (2020), 25--32.
%
%
\bibitem{Han}
Hanusch, M., \emph{Regularity of Lie groups},
Commun.\ Anal.\ Geom.\
{\bf 30} (2022),
%Number 1
53--152.
%
%
\bibitem{JL1}
Jafarpour, S. and A. D. Lewis,
``Time-Varying Vector Fields and Their Flows,''
Springer, 2014.
%
%
\bibitem{Kel}
Keller, H.~H., ``Differential Calculus
in Locally Convex Spaces'', Springer-Verlag,
Berlin, 1974.
%
%
\bibitem{Key}
Kelley, J.~L.,
``General Topology,'' Springer-Verlag,
New York, 1975.
%
%
\bibitem{KaS}
Klose, D. and F. Schuricht,
\emph{Parameter dependence for a class of ordinary
differential equations with measurable right-hand side},
Math. Nachr. 284 (2011),
%No. 4,
507--517.
%
%
\bibitem{KaM} Kriegl, A. and P.\,W. Michor,
``The Convenient Setting of Global Ana\-lysis,''
AMS, Providence, 1997.
%
%
\bibitem{Les}
Leslie, J.,
\emph{On the group of real analytic diffeomorphisms
of a compact real analytic manifold},
Trans.\ Am.\ Math.\ Soc.\ {\bf 274} (1982), 651--669.
%
%
\bibitem{Lew}
Lewis, A. D.,
\emph{Integrable and absolutely continuous vector-valued functions},
Rocky Mountain J. Math.\ {\bf 52} (2022),
%no. 3,
925--947.
%
%
\bibitem{Mic}
Michor, P.\,W., ``Manifolds of Differentiable Mappings'',
Shiva Publ., Orpington, 1980.
%
%
\bibitem{Mil} Milnor, J., \emph{Remarks on infinite-dimensional Lie groups},
pp.\,1007--1057 in: B.\,S. DeWitt and R. Stora (eds.),
``Relativit\'{e}, groupes et topologie II,'' North-Holland,
Amsterdam, 1984.
%
%
\bibitem{Nee}
Neeb, K.-H.,
\emph{Towards a Lie theory of locally convex groups},
Jpn.\ J. Math.\ {\bf 1} (2006), 291--468. 
%
%
\bibitem{NaS}
Neeb, K.-H. and H. Salmasian,
\emph{Differentiable vectors and unitary repre-
sentations of Fr\'{e}chet--Lie
supergroups}, Math.\ Z. {\bf 275} (2013),
%no. 1-2,
419--451.
%
%
\bibitem{Nik}
Nikitin, N.,
\emph{Regularity properties of infinite-dimensional Lie groups and exponential laws},
doctoral thesis, Paderborn University, 2021
(see\linebreak
%
https:/\!{}/nbn-resolving.de/urn:nbn:de:hbz:466:2-39133).
%
%
\bibitem{PaS}
Pressley, A. and G. Segal,
``Loop Groups,''
Clarendon Press, Oxforf, 1988.
%
%
\bibitem{Sch}
Schechter, E., ``Handbook of Analysis and its Foundations,''
Academic Press, San Diego, 1997.
%
%
\bibitem{Sm0}
Schmeding, A., \emph{The diffeomorphism group of a non-compact orbifold},
Dissertationes Math.\ {\bf 507} (2015), 179 pp.
%
%
\bibitem{Sme}
Schmeding, A.,
``An Introduction to Infinite-Dimensional Differential Geometry,''
Cambridge Univ.\ Press, Cambridge, 2023. 
%
%
\bibitem{SvM}
Schuricht, F. and H. von der Mosel,
\emph{Ordinary differential equations with measurable
right-hand side and parameters in metric spaces},
Preprint 676, SFB 256, Uni-Bonn, 2000
(cf.\ http://www.instmath.rwth-aachen.de/~heiko/veroeffentlichungen/ode.dvi).
%
%
\bibitem{Son}
Sontag, E. D.,
``Mathematical Control Theory,''
Springer, New York, ${}^2$1998. 
%
%
\bibitem{Tho}
Thomas, E.,
\emph{The Lebesgue-Nikodym theorem for vector valued Radon measures},
Mem.\ Am.\ Math.\ Soc.\ {\bf 139} (1974),
101 pp.
%
%
\bibitem{Wal}
Walter, B.,
\emph{Weighted diffeomorphism groups of Banach spaces and weighted mapping groups},
Diss.\ Math.\ {\bf 484} (2012), 126 pp.
%
%
\bibitem{WaB}
Whitney, H.
and F. Bruhat,
\emph{Quelques propri\'{e}t\'{e}s fondamentales des ensembles analytiques-r\'{e}els},
Comment.\ Math.\ Helv.\ {\bf 33} (1959), 132--160.
%
%
\end{thebibliography}
\end{document}